\documentclass[11pt]{amsart}
\usepackage[linkcolor=blue, citecolor = blue]{hyperref}
\hypersetup{colorlinks=true}

\usepackage{amsthm,amsfonts,amsmath,amssymb,amscd} 

\usepackage{verbatim}
\usepackage{float}
\usepackage{wrapfig}
\usepackage{mathtools}
\usepackage[color=orange!20, linecolor=orange]{todonotes}

\usepackage{enumitem}

\usepackage{ytableau}

\usepackage{tikz}
\usetikzlibrary{decorations.markings}
\tikzstyle{vertex}=[circle, draw, inner sep=0pt, minimum size=4pt]

\usetikzlibrary{arrows,automata}
\usepackage{tikz, tikz-3dplot, pgfplots}
\usepackage{tkz-graph}
\usetikzlibrary[positioning,patterns]

\newtheorem{theorem}{Theorem}[section]
\newtheorem{proposition}[theorem]{Proposition}
\newtheorem{lemma}[theorem]{Lemma}
 
\theoremstyle{definition}

\newtheorem{definition}[theorem]{Definition}

\newtheorem{problem}[theorem]{Problem}
\newtheorem{conjecture}[theorem]{Conjecture}

\theoremstyle{remark}
\newtheorem{remark}[theorem]{Remark}

\title[Grothendieck-positive specializations]{Positive specializations of symmetric Grothendieck polynomials} 

\author[Damir Yeliussizov]{Damir Yeliussizov}

\address{KBTU, Almaty, Kazakhstan}
\email{\href{mailto:yeldamir@gmail.com}{yeldamir@gmail.com}}

\begin{document}

\begin{abstract}
It is a classical fundamental result that Schur-positive specializations of the ring of symmetric functions are characterized via totally positive functions whose parametrization describes the Edrei--Thoma theorem.   
In this paper, we study positive specializations of symmetric Grothendieck polynomials,  $K$-theoretic deformations of Schur polynomials.  
\end{abstract}

\maketitle


\section{Introduction}
Characterization of {\it Schur-positive} specializations of the ring of symmetric functions relies on a fundamental result in the theory of {total positivity}, the famous Edrei--Thoma theorem \cite{edrei, thoma}. It became foundational to representation theory of the infinite symmetric group \cite{thoma, vk81, ker, bool}. In combinatorics, positive specializations of Schur functions provide a powerful tool in enumeration 
of plane partitions \cite{sta}. In probability, besides connections with the Plancherel measure on partitions \cite{vkp}, they are important in Schur processes \cite{okr}. 

\vspace{0.5em}

In Schubert calculus, Schur polynomials 
relate to the cohomology of Grassmannians (e.g. \cite{fult}). 
Similarly, $K$-theoretic Schubert calculus gives rise to certain deformations of Schur polynomials. 
These objects are known as {\it Grothendieck polynomials} \cite{las, fk}. They bear many similarities with the Schur basis but have significant differences. For instance, they are {inhomogeneous} as oppose to Schur functions. In {combinatorial $K$-theory}, underlying structures behind such deformations usually become more complicated, see e.g. \cite{buch, buchk, lp, pp, thomasyong2, dy, dy2}. 


\subsection{Overview of main results} In this paper, we study positivity of $K$-theoretic deformations of Schur polynomials. 
Namely, we study {\it positive specializations} for three classes of {\it symmetric Grothendieck polynomials}  (see definitions later):
\begin{itemize}
\item[(F1)] the `positive' basis $\{\tilde G_{\lambda} \}$
\item[(F2)] the usual basis $\{G_{\lambda} \}$ with alternating signs
\item[(F3)] the dual basis $\{g_{\lambda} \}$
\end{itemize}

A main feature of the family (F1) is that its structure constants and monomial expansions are all positive. It turns out that characterization of positive specializations for $\{\tilde G_{\lambda} \}$ is more difficult than for (F2).  
We are able to solve this problem for the families (F1), (F2) via transition to Schur-positive specializations. This might seem surprising, especially for (F1), as Schur functions expand non-positively in $\{\tilde G_{\lambda} \}$. We also show that positive specializations for (F1) characterize {\it boundary} of a certain {\it filtered Young graph}. 
As for the polynomials (F3), we describe a class of positive specializations and conjecture that it is complete. We also discuss how these specializations produce two analogues of the Plancherel measure 
on partitions: one is naturally related to  the {\it corner growth model} in probability (e.g. \cite{joh1, romik}), another is the {\it Plancherel-Hecke measure} \cite{thomasyong}.  

Let us now summarize the background and main results in more detail.

\subsection{The Edrei--Thoma theorem} A fundamental result in the theory of {\it total positivity} is the theorem proved by Edrei \cite{edrei} and Thoma \cite{thoma}, which characterizes totally nonnegative Toeplitz matrices. It was originally conjectured by Schoenberg \cite{scho}.

A matrix is called {\it totally nonnegative} if all its minors are nonnegative reals. 
A formal power series $A(z) = 1 + \sum_{n = 1}^{\infty} a_n z^n \in \mathbb{R}[[z]]$ is a {\it totally positive function} if the associated (infinite, upper triangular) {\it Toeplitz matrix} $[a_{i - j}]_{i,j \ge 0}$ is totally nonnegative (where $a_0 = 1$ and $a_n = 0$ for $n < 0$). The sequence $\{ a_n\}$ is then called {\it P\'olya frequency sequence} and every polynomial $A_N(z) = 1 + \sum_{n = 1}^{N} a_n z^n$ (for $N = 1, 2, \ldots$) has only negative 
real roots, which is the Aissen--Schoenberg--Whitney theorem \cite{asw}.

\begin{theorem}[Edrei--Thoma]\label{ets0}
$A(z) = 1 + \sum_{n = 1}^{\infty} a_n z^n$ 
is a totally positive function if and only if
$$
A(z) = e^{\gamma z} \prod_{n = 1}^{\infty} \frac{1 + \beta_n z}{1 - \alpha_n z}
$$
for nonnegative real parameters
$\{\alpha_n\}$,  $\{\beta_n\}$, 
and $\gamma$ such that $\sum_{n} (\alpha_n + \beta_n) < \infty$. 
\end{theorem}

Let $\Lambda$ be the ring of symmetric functions. One of the most important bases of $\Lambda$ is given by {\it Schur functions} $\{ s_{\lambda}\}$. 
A homomorphism (specialization) $\rho : \Lambda \to \mathbb{R}$ is called {\it Schur-positive} if $\rho(s_{\lambda}) \ge 0$ for all partitions $\lambda$.  
The ring $\Lambda$ is a polynomial ring with generators $\{ h_n \}$ of complete homogeneous symmetric functions. Thus any homomorphism of $\Lambda$ can be specified via values of $h_n$. Using the {\it Jacobi-Trudi identity} $s_{\lambda} = \det[h_{\lambda_i - i +j}],$ 
the Edrei-Thoma theorem equivalently characterizes Schur-positive specializations as follows.

\begin{theorem}\label{ets}
A homomorphism $\rho : \Lambda \to \mathbb{R}$ is Schur-positive if and only if 
\begin{equation*}\label{hz}
\rho(H(z)) := 
1 + \sum_{n = 1}^{\infty} \rho(h_n)\, z^n  = e^{\gamma z} \prod_{n = 1}^{\infty} \frac{1 + \beta_n z}{1 - \alpha_n z}
\end{equation*}
for nonnegative reals 
$\{\alpha_n\}$,  $\{\beta_n\}$, 
and $\gamma$
such that $\sum_{n} (\alpha_n + \beta_n) < \infty$. 
Equivalently, $\rho(H(z))$ is a totally positive function, where $H(z)$ is defined in \eqref{hz}.
\end{theorem}

An important specialization of $\Lambda$ is the {\it Plancherel specialization} $\pi$ which corresponds to the parameters $\gamma = 1,$ $\alpha_k = \beta_k = 0,$ and for which $\pi(s_{\lambda}) = f^{\lambda}/{n!}$, where $f^{\lambda}$ is the number of standard Young tableaux (SYT) of shape $\lambda \vdash n$, or the dimension of an irreducible representation (indexed by $\lambda$) of the symmetric group $S_n$.

A relationship between totally positive functions and characters of the infinite symmetric group $S_{\infty}$ was made by Thoma \cite{thoma}. Vershik and Kerov \cite{vk81}  
interpreted the parameters $\alpha_n, \beta_n$ asymptotically as normalized row and column lengths in growing partitions, describing characters of $S_\infty$. 
Original proofs of Edrei and Thoma used deep results from complex analysis, whereas Vershik and Kerov's approach relied on asymptotic representation theory. 
Note also that the difficult part of Theorem~\ref{ets}(or \ref{ets0}) is the {\it only if} direction, i.e. if a homomorphism is positive then it satisfies a given parametrization. 

Similar positivity results are known for other classes of symmetric functions and generalizations of Schur polynomials, such as 
Jack symmetric functions \cite{koo}, Schur-$P,Q$ functions \cite{naz}, some general settings in \cite{bo00}, and a recent solution of Kerov's conjecture on positive specializations of Macdonald polynomials \cite{matveev}. We refer to \cite{bool} for more on the Edrei-Thoma theorem and representation theory of the infinite symmetric group, including background and many references therein.



\subsection{Symmetric Grothendieck polynomials} 
{\it Symmetric} (or {\it stable}) {\it Grothendieck polynomials} 
are considered as a $K$-theoretic deformation of Schur polynomials. They were first studied by Fomin and Kirillov \cite{fk}. We begin by defining {\it positive} symmetric Grothendieck polynomials $\{\tilde G^{}_{\lambda} \}$ by the following combinatorial formula due to Buch~\cite{buch}:
$$\tilde G^{}_{\lambda} = \tilde G^{}_{\lambda}(x_1, x_2, \ldots) := \sum_{T \in SVT(\lambda)} 
\prod_{i \ge 1} x_i^{\# i \text{ in } T},$$
where the sum runs over shape $\lambda$ {\it set-valued tableaux} (SVT), a generalization of semistandard Young tableaux (SSYT) so that boxes may contain sets of integers (for precise definitions see Sec.~\ref{gback}).
One can see that 
$\tilde G^{}_{\lambda} = s_{\lambda} + \{\text{higher degree terms}\} \in \hat\Lambda,$
where $\hat\Lambda$ is the completion of 
$\Lambda$, that includes 
{\it infinite} linear combinations of basis elements.
For example,
$$
\tilde G^{}_{(1)} = e_1 + e_2 + e_3 + \ldots \quad\text{ or }\quad 1 + \tilde G^{}_{(1)} = \prod_{n = 1}^{\infty} (1 + x_n) 
$$
where $e_k$ is the $k$th elementary symmetric function.

Crucially, for all partitions $\mu, \nu$, the product
\begin{equation}\label{finitexp}
\tilde G^{}_{\mu} \cdot \tilde G^{}_{\nu} = \sum_{\lambda} 
c^{\lambda}_{\mu \nu}\, \tilde G^{}_{\lambda}, \qquad |\lambda| \ge |\mu| + |\nu|, \quad c^{\lambda}_{\mu \nu} \in \mathbb{Z}_{\ge 0}
\end{equation}
expands as a {\it finite} sum. The nonnegative integers $c^{\lambda}_{\mu \nu}$ are {\it generalized Littlewood-Richardson} (LR) coefficients as in the lowest degree case $|\lambda| = |\mu| + |\nu|$ they become just the usual LR coefficients corresponding to product of Schur functions. This was proved by Buch \cite{buch} via an explicit combinatorial LR rule for $c^{\lambda}_{\mu \nu}$. This finite expansion property can also be seen without any combinatorial interpretation of $c^{\lambda}_{\mu \nu}$ but conceptually on symmetric functions level \cite{dy2}. 

In light of the multiplication rule \eqref{finitexp} one defines the commutative ring 
$$\Gamma := \bigoplus_{\lambda} \mathbb{R} \cdot \tilde G^{}_{\lambda}$$ 
with a formal basis $\{ \tilde G^{}_{\lambda}\}$. 
The ring $\Gamma$ is related to $K$-theory of Grassmannians and we refer to \cite{buch} for background. 
As is also mentioned there, the ring $\Gamma$ has somewhat unclear structure. For instance, it is {not} isomorphic to $\Lambda$ (but as completions $\hat\Lambda \cong \hat\Gamma$); we also do not know if it is a polynomial ring. Buch  conjectured that any $\tilde G^{}_{\lambda}$ is a polynomial in the elements $\tilde G^{}_{R}$ for rectangular partitions $R \subset \lambda$; this would imply that the localization ring generated by $\Gamma$ and $q = (1 +  \tilde G^{}_{(1)})^{-1}$ is generated by the elements $\{\tilde G^{}_{(n)}\},$ $\{\tilde G^{}_{(1^{n})}\},$ and $q$ \cite{buch}.

\subsection{Grothendieck-positive specializations}
A homomorphism $\varphi : \Gamma \to \mathbb{R}$ is called {\it $G$-positive} (or {\it Grothendieck-positive}) 
if 
$$\varphi(\tilde G^{}_{\lambda}) \ge 0 \text{ for all partitions } \lambda.$$ 
Say that $\varphi$ is {\it normalized} if $\varphi(\tilde G_{(1)}) = 1$. We shall usually use the notation $\tilde G_{\lambda}(\varphi)$  for $\varphi(\tilde G_{\lambda})$.

\

The main problem that we address and solve in this paper is the following.
\begin{problem}
Describe $G$-positive specializations of the ring $\Gamma$.
\end{problem}

One of the key results that we prove is the 
following transition theorem from Grothendieck to Schur positive specializations.
\begin{theorem}\label{t1}
Let $\varphi : \Gamma \to \mathbb{R}$ be a $G$-positive homomorphism of $\Gamma$. Then the map $\rho$ given by  
$$
\rho: h_n \longmapsto \frac{\tilde G^{}_{(n)}(\varphi) + \tilde G^{}_{(n+1)}(\varphi)}{1 + \tilde G^{}_{(1)}(\varphi) } \qquad n = 1,2, \ldots
$$
defines a Schur-positive specialization of $\Lambda$.
\end{theorem}

Alternatively, the theorem states that we have the generating function for the elements $\{\tilde G_{(n)}(\varphi) \}$: 
\begin{equation*}\label{eq2}
1 + (z + 1) \sum_{n = 1}^{\infty} \tilde G^{}_{(n)}(\varphi)\, z^{n - 1} = (1 + \delta)\, e^{\gamma z} \prod_{n = 1}^{\infty} \frac{1 + \beta_n z}{1 - \alpha_n z}
\end{equation*}
for nonnegative reals $\{\alpha_n\}$,  $\{\beta_n\}$, $\gamma$, and $\delta$ such that $\sum_{n} (\alpha_n + \beta_n) < \infty$.
Here $\tilde G^{}_{(1)}(\varphi) = \delta$ and for $z = -1$ (given it converges) we additionally get 
$$
\delta = -1 + e^{\gamma} \prod_{n = 1}^{\infty} \frac{1 + \alpha_n}{1 - \beta_n}
$$
To prove this theorem we define certain auxiliary functions via Jacobi-Trudi-type formula (Sec.~\ref{gps}) and show that they are Grothendieck-positive, i.e. expand positively in the Grothendieck basis. 
  
More precisely, we characterize normalized $G$-positive specializations as follows.
\begin{theorem}\label{t2}
Let $\varphi : \Gamma \to \mathbb{R}$ be a normalized $G$-positive homomorphism. Then the map
$$
\rho: h_n \longmapsto \frac{1}{2}\left(\tilde G^{}_{(n)}(\varphi) + \tilde G^{}_{(n+1)}(\varphi)\right) \qquad n = 1,2, \ldots
$$
defines a Schur-positive specialization of $\Lambda$ parametrized by
nonnegative reals 
$\{\alpha_n\}$,  $\{\beta_n < 1\}$, $\gamma$ 
such that 
\begin{equation*}
\gamma = \log 2 - \sum_{n} \log (1 + \alpha_n) + \sum_{n} \log (1 - \beta_n). 
\end{equation*}
Conversely, given a Schur-positive specialization parametrized by $\{\alpha_n\}$,  $\{\beta_n\}$, $\gamma$ as above, it extends to a normalized $G$-positive homomorphism of $\Gamma$. 
\end{theorem}

\subsection{Harmonic functions and boundary of a filtered Young's graph}  
Consider the (infinite) {\it filtered Young graph} $\widetilde{\mathbb{Y}^{}}$ defined as follows:
\begin{itemize}
\item[(i)] its vertices are labeled by partitions $\lambda$;
\item[(ii)] there is an arc $\lambda \to \mu$ iff $\mu/\lambda$ is a {\it rook strip} (i.e. no two boxes lie in the same row or column)
\end{itemize}
As we will see, this graph 
encodes Pieri rules for Grothendieck polynomials.

Let $\mathcal{P}$ be the set of partitions. 
A function $\varphi : \mathcal{P} \to \mathbb{R}_{\ge 0}$ is called {\it harmonic}\footnote{Harmonic functions in Vershik--Kerov sense \cite{vk81}, usually defined for graded graphs \cite{ker, bo00}.} on the graph $\widetilde{\mathbb{Y}}$ if 
$$
\varphi(\varnothing) = 1\quad \text{ and }\quad \varphi(\lambda) = \sum_{\mu\, :\, \lambda \to \mu} 
\varphi(\mu).
$$

Let $H^{}(\widetilde{\mathbb{Y}})$ be the {convex} set of harmonic functions on $\widetilde{\mathbb{Y}^{}}$ 
and let $\partial\widetilde{\mathbb{Y}} \subset H^{}(\widetilde{\mathbb{Y}})$ be the set of {\it extreme points} of $H^{}(\widetilde{\mathbb{Y}})$, or the {\it boundary} of $\widetilde{\mathbb{Y}}$ (i.e. the set of harmonic functions that are not expressible as nontrivial convex combinations of other harmonic functions). 

For every function $\varphi \in H^{}(\widetilde{\mathbb{Y}})$ define the {linear functional} $\hat\varphi : \Gamma \to \mathbb{R}$ such that $\hat\varphi(\tilde G_{\lambda}) := \varphi(\lambda)$.
Then $G$-positive homomorphisms characterize the corresponding boundary. Namely, the set of linear functionals of the  
boundary $\partial \widetilde{\mathbb{Y}}$ coincides with the set of normalized $G$-positive homomorphisms. We have the following version of Vershik--Kerov ``ring theorem" \cite{vk, bool} for $\widetilde{\mathbb{Y}}$.

\begin{theorem}
We have: $\varphi \in \partial \widetilde{\mathbb{Y}}$ i.e. $\varphi$ is extreme if and only if 
$\hat\varphi$ is a normalized $G$-positive homomorphism of $\Gamma$.
\end{theorem}

The graph $\widetilde{\mathbb{Y}}$ is viewed as a {\it filtered} deformation of graded Young's lattice $\mathbb{Y}$ as its vertices $\lambda \in \mathcal{P}$ form a {\it graded} set ranked by $|\lambda|$ and arcs join vertices from lower to higher ranks. Note also that this graph is part of the {\it M\"obius deformation} of Young's graph \cite{pp} (cf. \cite{dy2}), which is one of major examples of {\it dual filtered graphs} studied by Patrias and Pylyavskyy \cite{pp} as $K$-theoretic analogues of Fomin's dual graded graphs \cite{fomin} and Stanley's differential posets \cite{stadiff}. 

\subsection{Grothendieck polynomials with alternating signs} The functions $\{G_{\lambda}\}$ are usually defined with alternating signs in monomials \cite{lenart, buch, lp}, namely as the functions 
$$
{G}_{\lambda}(x_1, x_2, \ldots) := (-1)^{|\lambda|}\, \tilde G_{\lambda}(-x_1, -x_2, \ldots).
$$
We also define {\it $\overline{G}$-positive homomorphisms} $\varphi$ 
of $\Gamma$ satisfying $\varphi({G}_{\lambda}) \ge 0$.
In Sec.~\ref{galt} we study and characterize 
such {$\overline{G}$-positive specializations} similarly as for $G$-positivity.

\subsection{Dual Grothendieck polynomials} 
There is a basis $\{ g^{}_{\lambda}\}$ of $\Lambda$ that is {\it dual} to $\{ {G}^{}_{\lambda}\}$ via the Hall inner product for which Schur functions form an orthonormal basis. It was explicitly described via plane partitions by Lam and Pylyavskyy in \cite{lp}. 
In Sec.~\ref{dgs} we also define and describe a class of {\it $g$-positive specializations} of the ring $\Lambda$ (Proposition~\ref{ggg}). We conjecture that this class is in fact complete, i.e. it describes all $g$-positive specializations. 

\section{Schur-positive specializations}
\subsection{Partitions and Young diagrams} 
A {\it partition} is a sequence $\lambda = (\lambda_1 \ge \ldots \ge \lambda_{\ell} > 0)$, where $\ell = \ell(\lambda)$ is the length of $\lambda$. Any partition can be represented as a {\it Young diagram} with $\lambda_i$ boxes in row $i$; equivalently, as the set $\{(i,j) : 1 \le i \le \ell, 1 \le j \le \lambda_i \}$. The partition $\lambda'$ is the {\it conjugate} of $\lambda$ obtained by transposing its diagram.  
We use English notation for drawing Young diagrams, index columns from left to right and rows from top to bottom. Let $\mathcal{P}$ be the set of partitions.

\subsection{Specializations of the ring of symmetric functions}
The ring $\Lambda$ of symmetric functions in the variables $\mathbf{x} = (x_1, x_2, \ldots)$ can be viewed as 
$$\Lambda \cong \mathbb{R}[h_1, h_2, \ldots] \cong \mathbb{R}[e_1, e_2, \ldots] \cong \mathbb{R}[p_1, p_2, \ldots]$$ i.e. 
a polynomial ring with one of the following sets of generators:
$$h_n := \sum_{1 \le i_1 \le \ldots \le i_n } x_{i_1} \cdots x_{i_n}, \qquad e_n := \sum_{1 \le i_1 < \ldots < i_n } x_{i_1} \cdots x_{i_n}, \qquad p_n := \sum_{1 \le i} x_i^{n}$$ of {complete homogeneous}, 
{elementary}, 
or {power sum symmetric functions}. 

A {\it specialization} is any homomorphism $\Lambda \to \mathbb{R}$ and it can be defined by specifying generators. For any specializations $\rho_1, \rho_2 : \Lambda \to \mathbb{R}$ we can define their {\it union} 
$$\rho = (\rho_1, \rho_2) = \rho_1 \cup \rho_2$$ 
via the power sum $\{ p_n\}$ generators as follows:
$$
\rho(p_{n}) := \rho_1(p_n) + \rho_2(p_n)  \text{ for all } n = 1, 2, \ldots
$$
Note that we have
\begin{align}\label{hz}
H(z) 	:= 1 + \sum_{n = 1}^{\infty} h_n\, z^n 
	= \prod_{n = 1}^{\infty} \frac{1}{1 - z x_n}
	= \exp\left(\sum_{n = 1}^{\infty} \frac{p_n}{n} z^n \right)
\end{align}
and hence 
$$
\rho(H(z)) = \rho_1(H(z)) \cdot \rho_2(H(z)).
$$
From these generating function identities we have the following.
\begin{lemma} \label{fact}
Let $\rho, \rho_1, \rho_2 : \Lambda \to \mathbb{R}$ be specializations of $\Lambda$.
We have: $\rho = (\rho_1, \rho_2)$ 
if and only if 
$\rho(H(z)) = \rho_1(H(z)) \cdot \rho_2(H(z)).$
\end{lemma}

\subsection{Schur-positive specializations}\label{sgen}
The ring $\Lambda$ has a linear basis $\{s_{\lambda} \}$ of Schur functions that can be defined as follows
$$
s_{\lambda}(x_1, x_2, \ldots) := \sum_{T \in SSYT(\lambda)} x^{T},
$$
where $SSYT(\lambda)$ is the set of {\it semistandard Young tableaux} (SSYT) of shape $\lambda$, i.e. filings of the boxes of the Young diagram of $\lambda$ with positive integers weakly increasing in rows (from left to right) and strictly increasing in columns (from top to bottom); and $x^{T} = \prod_{i \ge 1} x_i^{a_i}$, where $a_i$  is the number of $i$'s in tableau $T \in SSYT(\lambda)$.

There is a {\it standard involutive automorphism} $\omega : \Lambda \to \Lambda$ given on generators by $\omega : h_n \mapsto e_n$ for all $n \ge 1$ and for which $\omega(s_{\lambda}) = s_{\lambda'}$.

\begin{definition}
A homomorphism $\rho : \Lambda \to \mathbb{R}$ is called {\it Schur-positive} if $\rho(s_{\lambda}) \ge 0$ for all $\lambda$.
\end{definition}

\begin{definition}
Let $\alpha, \beta, \gamma \ge 0$ be scalars. Define the following {{\textit{Schur-positive  generators}}} $\phi, \varepsilon, \pi$ given by:
\begin{itemize}
\item The specialization $\phi_{\alpha}$:  
$$\phi_\alpha(H(z)) = (1 - \alpha z)^{-1}$$ for which $\phi_\alpha (s_{\lambda}) = \alpha^{|\lambda|}$ if $\ell(\lambda) = 1$ and $0$ otherwise. Equivalently, it is just single variable substitution $x_1 \mapsto \alpha$, $x_k \mapsto 0$ for $k \ge 2$.
\item The specialization $\varepsilon_{\beta}$: 
$$\varepsilon_\beta (H(z)) = 1 + \beta z$$ for which $\varepsilon_\beta (s_{\lambda}) = \beta^{|\lambda|}$ if $\lambda_1 = 1$ and $0$ otherwise. Note that $\varepsilon_\beta = \phi_\beta \circ \omega$, i.e. a composition of $\phi$ with the involution $\omega$.
\item The {\it Plancherel specialization} $$\pi_\gamma (H(z)) = e^{\gamma z}$$ for which $\pi_\gamma (s_{\lambda}) = \gamma^{n} f^{\lambda}/n!$ for $\lambda \vdash n$, where $f^{\lambda}$ is the number of SYT of shape $\lambda$. Equivalently, it is given by $p_1 \mapsto \gamma$ and $p_k \mapsto 0$ for $k \ge 2$.  
\end{itemize}
\end{definition}

Note that union of Schur-positive specializations is also Schur-positive. This can be seen from the branching formulas 
$$
\rho(s_{\lambda}) = \sum_{\mu} \rho_1(s_{\lambda/\mu})\, \rho_2(s_{\mu}) = \sum_{\mu, \nu} c^{\lambda}_{\mu \nu}\, \rho_1(s_{\nu})\, \rho_2(s_{\mu}) \ge 0, \quad 
$$
if $\rho =(\rho_1, \rho_2) = \rho_1 \cup \rho_2$ is a union of Schur-positive specializations. 

Combining these facts with an observation made in Lemma \ref{fact}, we obtain the following equivalent formulation for characterization (Theorem~\ref{ets}) of Schur-positive specializations.
\begin{theorem}[Edrei-Thoma factorization form] \label{etfact}
A specialization $\rho : \Lambda \to \mathbb{R}$ is Schur-positive if and only if 
$$\rho = \pi_\gamma \cup \phi_{\alpha_1} \cup \phi_{\alpha_2} \cup \dots \cup \varepsilon_{\beta_1} \cup \varepsilon_{\beta_2} \cup \dots$$
for nonnegative reals 
$\gamma$, $\{\alpha_n\}$, $\{\beta_n \}$ such that $\sum_{n} (\alpha_n + \beta_n) < \infty$.
\end{theorem}

\begin{remark}
In terms of the power sum generators, Schur-positive specializations $\rho$ are given by 
$$
\rho(p_k) = \pi_\gamma(p_k) + \sum_{n} \phi_{\alpha_n}(p_k) + \sum_{n} \varepsilon_{\beta_n}(p_k) \text{ for } k \ge 1,
$$
or more precisely we have
\begin{align*}
\rho &: p_1 \longmapsto \gamma + \sum_{n} (\alpha_n + \beta_n)\quad \text{ and }\quad \rho : p_k \longmapsto \sum_{n} \left((\alpha_n)^k + (-1)^{k-1}(\beta_n)^k\right) \text{ for } k \ge 2.
\end{align*}
The condition $\sum_{n} (\alpha_n + \beta_n) < \infty$ is necessary and sufficient for all $p_k$ to converge.
\end{remark}

\begin{remark}
Product of generating functions that specialize $H(z)$ corresponds to product of associated Toeplitz matrices so that $(1 - \alpha z)^{-1}$, $(1 + \beta z)$, $e^{\gamma z}$ generate totally positive functions $\rho(H(z))$.  
\end{remark}

\section{Symmetric Grothendieck polynomials}\label{gback}
Let us give some notation for Young diagrams. We use definitions introduced in \cite{dy2}. 

 \ytableausetup{smalltableaux}
Denote by $I(\lambda)$ the set of {\it inner corner} boxes of $\lambda$ and $i(\lambda) = \# I(\lambda)$. For partitions $\lambda \supset \mu$ define the following extension of skew shapes: $$\lambda/\!\!/\mu := \lambda/\mu \cup I(\mu).$$ E.g. $(5331)/\!\!/(432)$  consists of the skew shape $(5331)/(432) = \{\scriptsize\ydiagram[*(lightgray)]{1}\}$  and $I(432) = \{\scriptsize\ydiagram[*(lightgray)\bullet]{1} \}$, see the picture.

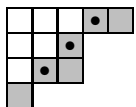
\begin{center}
 \ytableausetup{smalltableaux}
 \begin{ytableau}
~ & ~ & ~ & *(lightgray)\bullet & *(lightgray)\\
~ & ~ & *(lightgray)\bullet \\
~ & *(lightgray){\bullet} & *(lightgray) \\
*(lightgray)
\end{ytableau}
\end{center}

Denote by $a(\lambda/\!\!/\mu)$ the number of {\it open boxes} of $I(\mu)$ that do not lie in the same column with any box of $\lambda/\mu$. Equivalently, it is the number of columns of $\lambda/\!\!/\mu$ that are not columns of $\lambda/\mu$. For example, $a((5331/\!\!/(432))) = 2$, the boxes $\ydiagram[*(lightgray)\text{o}]{1}$ in the picture.

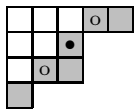
\begin{center}
\begin{ytableau}
~ & ~ & ~ & *(lightgray)\text{o} & *(lightgray)\\
~ & ~ & *(lightgray)\bullet \\
~ & *(lightgray)\text{o} & *(lightgray) \\
*(lightgray)
\end{ytableau}
\end{center}

Denote by $c(\lambda/\mu)$ and $r(\lambda/\mu)$ the number of (nonempty) columns and rows of $\lambda/\mu$, respectively. 

We say that $\lambda/\mu$ is a {\it horizontal} (resp. {\it vertical}) {\it strip} if $\lambda/\mu$ has no two boxes in the same column (resp. row). Say that $\lambda/\mu$ is a {\it rook strip} if no two boxes  of $\lambda/\mu$ lie in the same row and column, equivalently when $|\lambda/\mu| = c(\lambda/\mu) = r(\lambda/\mu)$, e.g.  $(5331)/(432)$ is a rook strip, see Fig.~\ref{svtfig} gray boxes.

\begin{definition}
A {\it set-valued tableau} (SVT) of shape $\lambda/\!\!/\mu = \lambda/\mu \cup I(\mu)$ is a filling of the boxes of $\lambda/\!\!/\mu$ by {\it sets} of positive integers such that 
\begin{itemize}
\item numbering is {\it semistandard}, i.e. if we replace each set by any of its elements, then the numbers increase from left to right and from top to bottom;
\item each box of $\lambda/\mu$ contains a non-empty set;
\item each box of $I(\mu)$ contains a set (maybe empty).
\end{itemize}
Let $SVT(\lambda/\!\!/\mu)$ be the set of SVT of shape $\lambda/\!\!/\mu$. For $T \in SVT(\lambda/\!\!/\mu)$, define the corresponding monomial $x^{T} = \prod_{i \ge 1} x_i^{a_i}$, where $a_i$ is the number of $i$'s in $T$. 
See the example in Fig.~\ref{svtfig}. 
\begin{figure}
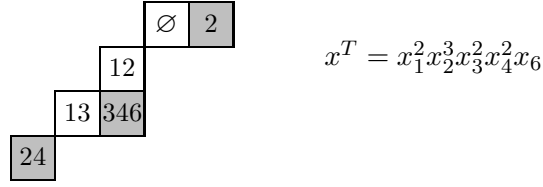

\begin{center}
{
 \ytableausetup{boxsize=normal}
 \begin{ytableau}
\none & \none & \none & \varnothing & *(lightgray) $\small 2$\\
\none & \none & $\small 12$ \\
\none & $\small 13$ & *(lightgray) ${\small 346}$ \\
*(lightgray) $\small 24$
\end{ytableau}
}
\qquad 
\begin{tabular}{c}
\\
$x^T = x_1^2 x_2^3 x_3^2 x_4^2 x_6$ \\ 
\end{tabular}
\end{center}
\caption{\small A set-valued tableau of shape $(5331)/\!\!/(432) = (5331)/(432) \cup I(432)$. The gray boxes is the skew shape $(5331)/(432)$; the white boxes is $I(432)$, corners of $(432)$. 
It is allowed to put $\varnothing$ in $I(\mu)$ but boxes of $\lambda/\mu$ must be nonempty.}
\label{svtfig}
\end{figure}
\end{definition}

We use the notation $\mathbf{x} = (x_1, x_2, \ldots)$ and $\mathbf{y} = (y_1, y_2, \ldots)$ for (infinite) sets of variables. 

\begin{definition} {\it Symmetric Grothendieck polynomials} $\tilde G^{}_{\lambda/\!\!/\mu}$ are defined via the following series: 
$$
\tilde G^{}_{\lambda/\!\!/\mu}(\mathbf{x}) := \sum_{T \in SVT(\lambda/\!\!/\mu)} 
x^T
$$
\end{definition}
In particular, $\tilde G^{}_{\lambda/\!\!/\varnothing} = \tilde G_{\lambda}$ for straight shapes. 
It is easy to see that 
$$\tilde G_{(1)} = -1 + \prod_{n}(1 + x_n) = e_1 + e_2 + \ldots$$
Note also that $$\tilde G_{\lambda/\!\!/\lambda} = \prod_{n}(1 + x_n)^{i(\lambda)} \ne \tilde G_{\varnothing} = 1.$$
\begin{proposition}[\cite{dy2}] \label{prop1}
The following branching formula holds:
\begin{equation}\label{br}
\tilde G^{}_{\lambda/\!\!/\mu}(\mathbf{x}, \mathbf{y}) = \sum_{\nu} \tilde G^{}_{\lambda/\!\!/\nu}(\mathbf{x}) \tilde G^{}_{\nu/\!\!/\mu}(\mathbf{y})
\end{equation}
For a single variable $x$ we have
\begin{equation}\label{single}
\tilde G^{}_{\lambda/\!\!/\mu}(x) = 
	\begin{cases}
		(1 + x)^{a(\lambda/\!\!/\mu)} x^{|\lambda/\mu|}, & \text{ if $\lambda/\mu$ is a horizontal strip;}\\
		0, & \text{ otherwise.} 
	\end{cases}
\end{equation}
\end{proposition}

We also have the following {\it Pieri rules} first proved by Lenart \cite{lenart}
\begin{equation}\label{gpieri0}
\tilde G_{(k)} \cdot \tilde G_{\lambda} = \sum_{\mu/\lambda \text{ hor. strip}} \binom{r(\mu/\lambda) - 1}{|\mu/\lambda| - k} \tilde G_{\mu} 
\end{equation}
and in particular the simple Pieri rule
$$
\tilde G_{(1)} \cdot \tilde G_{\lambda} = \sum_{\mu/\lambda \text{ rook strip}} \tilde G_{\mu}
$$

We have $\tilde G_{\lambda} = s_{\lambda} + \{\text{higher degree terms}\} \in \hat\Lambda$ are elements of the completion $\hat\Lambda$ of the ring $\Lambda$ consisting of unbounded degree elements. Since the lowest degree component is the Schur basis $\{s_{\lambda}\}$, the functions $\{\tilde G_{\lambda}\}$ are linearly independent. To describe transition coefficients with the Schur basis let us introduce two types of tableaux.  
\begin{definition}[\cite{lenart, lp}]\label{eleg}
Define the following tableaux:
\begin{itemize}
\item 
A {\it strict elegant tableau} of shape $\mu/\lambda$ is an SYT (i.e. strictly increasing in rows and columns)  
whose entries in row $i$ lie in $[1, i-1]$ for all $i$. Let $\mathrm{r}_{\mu/\lambda}$ be the number of strict elegant tableaux of shape $\mu/\lambda$. In particular,  if $\mathrm{r}_{\mu/\lambda} > 0$ then $\mu_1 = \lambda_1$.
\item An {\it elegant tableau}  of shape $\mu/\lambda$ is an SSYT 
whose entries in row $i$ lie in $[1, i - 1]$ for all $i$. 
Let $\mathrm{f}_{\mu/\lambda}$ be the number of elegant tableaux of shape $\mu/\lambda$. In particular, if $\mathrm{f}_{\mu/\lambda} > 0$ then $\mu_1 = \lambda_1$.
\end{itemize}
\end{definition}
Then we have the following infinite expansions \cite{lenart}
\begin{equation}\label{gschur}
\tilde G_{\lambda} = \sum_{\mu \supset \lambda} \mathrm{r}_{\mu/\lambda}\, s_{\mu}, \qquad\qquad s_{\lambda} = \sum_{\mu \supset \lambda} (-1)^{|\mu/\lambda|} \mathrm{f}_{\mu/\lambda}\, \tilde G^{}_{\mu}.
\end{equation}

As discussed in the introduction, for all $\mu, \nu$ the product 
\begin{equation}\label{prodg}
\tilde G_{\mu} \cdot \tilde G_{\nu} = \sum_{\lambda} c^{\lambda}_{\mu \nu}\, \tilde G_{\lambda}, \quad |\lambda| \ge |\mu| + |\nu|, \quad c^{\lambda}_{\mu \nu} \in \mathbb{Z}_{+}
\end{equation}
expands as a finite sum; here $c^{\lambda}_{\mu \nu}$ are generalized LR coefficients \cite{buch}. Hence we can define the commutative ring 
$$
\Gamma := \bigoplus_{\lambda} \mathbb{R} \cdot \tilde G_{\lambda}
$$
with a formal basis $\{ \tilde G_{\lambda}\}$ and the product given by \eqref{prodg}. There is an involutive automorphism $\tau: \Gamma \to \Gamma$ given by $\tau : \tilde G^{}_{\lambda} \mapsto \tilde G^{}_{\lambda'}$ (see \cite{buch, dy}).
We have $\tau (\tilde G_{\lambda/\!\!/\mu}) = \tilde G_{\lambda'/\!\!/\mu'}$ (see \cite{dy2}).

We also have finite expansions (see \cite{buch}) 
\begin{equation}\label{gbr}
\tilde G_{\lambda}(\mathbf{x}, \mathbf{y}) = \sum_{\mu, \nu} d^{\lambda}_{\mu \nu}\, \tilde G_{\mu}(\mathbf{x})\, \tilde G_{\nu}(\mathbf{y}), \quad d^{\lambda}_{\mu \nu} \in \mathbb{Z}_{+}
\end{equation} 
\begin{equation}\label{pprod}
\tilde G_{\lambda/\!\!/\mu} = \sum_{\nu} d^{\lambda}_{\mu \nu}\, \tilde G_{\nu} \in \Gamma. 
\end{equation}
Note that the structure coefficients are the same in both expansions which follows from branching formulas for $\tilde G_{\lambda}$ (or Hopf-algebraic properties). 

\subsection{Specializations of the ring $\Gamma$}
A {\it specialization} of $\Gamma$ is any homomorphism $\Gamma \to \mathbb{R}$.
\begin{definition}\label{uu}
Let $\varphi_1, \varphi_2$ be homomorphisms of $\Gamma$. Define their {\it union} $\varphi = (\varphi_1, \varphi_2) = \varphi_1 \cup \varphi_2$ such that for all $\lambda$ we have
$$
\tilde G_{\lambda}(\varphi) = \tilde G_{\lambda}(\varphi_1, \varphi_2) = \sum_{\nu} \tilde G^{}_{\lambda/\!\!/\nu}(\varphi_1)\, \tilde G^{}_{\nu}(\varphi_2).
$$
\end{definition}
By the branching formula \eqref{br}, it is well-defined and is compatible with union of specializations of $\Lambda$ as we discuss below. 
\begin{definition}
Let $\rho$ be a specialization of $\Lambda$. Say that $\rho$ {\it extends} to a specialization of $\Gamma$ if $\tilde G_{\lambda}$ are well-defined as an image under $\rho$. Formally, if the homomorphism $\hat\rho : \Gamma \to \mathbb{R}$ given by 
\begin{equation}
\tilde G_{\lambda}(\hat\rho) := \sum_{\mu} \mathrm{r}_{\mu/\lambda}\, \rho(s_{\mu}) 
\end{equation}
is well-defined. In particular, all such infinite sums converge. 
\end{definition}
\begin{lemma}
Let $\rho_1, \rho_2$ be specializations of $\Lambda$ that {\it extend} to specializations $\hat\rho_1, \hat\rho_2$ of $\Gamma$. Then the union $\rho = (\rho_1, \rho_2)$ of $\Lambda$ extends to the specialization $\hat\rho = (\hat\rho_1, \hat\rho_2)$ of $\Gamma$. 
\end{lemma}
\begin{proof}
Follows from the branching formulas \eqref{br}, \eqref{gbr} and Schur expansions \eqref{gschur}. Namely, we have the following identities
\begin{align}\label{gxy1}
	\tilde G_{\lambda}(\mathbf{x}, \mathbf{y}) 
		= \sum_{\mu, \nu} d^{\lambda}_{\mu \nu}\, \tilde G_{\mu}(\mathbf{x})\, \tilde G_{\nu}(\mathbf{y})
		= \sum_{\mu, \nu} d^{\lambda}_{\mu \nu}\, \sum_{\eta} \mathrm{r}_{\eta/\mu}\, s_{\eta}(\mathbf{x}) \sum_{\kappa} \mathrm{r}_{\kappa/\nu}\, s_{\kappa}(\mathbf{y}).
\end{align}
On the other hand,
\begin{align}\label{gxy2}
	\tilde G_{\lambda}(\mathbf{x}, \mathbf{y}) 
		= \sum_{\theta} \mathrm{r}_{\theta/\lambda}\, s_{\theta}(\mathbf{x}, \mathbf{y})
		= \sum_{\theta}  \mathrm{r}_{\theta/\lambda}\, \sum_{\eta, \kappa} c^{\theta}_{\eta \kappa}\, s_{\eta}(\mathbf{x})\, s_{\kappa}(\mathbf{y}).
\end{align}
Combining these identities by applying $\rho$ with the definition of unions we obtain
\begin{align*}
	\tilde G_{\lambda}(\hat\rho) 
		&= \sum_{\mu, \nu} d^{\lambda}_{\mu \nu}\, \tilde G_{\mu}(\hat\rho_1)\, \tilde G_{\nu}(\hat\rho_2)\\
		(\text{by \eqref{gxy1}})\quad &= \sum_{\mu, \nu} d^{\lambda}_{\mu \nu}\, \sum_{\eta} \mathrm{r}_{\eta/\mu}\, \rho_1(s_{\eta}) \sum_{\kappa} \mathrm{r}_{\kappa/\nu}\, \rho_2(s_{\kappa})\\
		(\text{by \eqref{gxy2}})\quad &= \sum_{\theta}  \mathrm{r}_{\theta/\lambda}\, \sum_{\eta, \kappa} c^{\theta}_{\eta \kappa}\, \rho_1(s_{\eta})\, \rho_2(s_{\kappa})\\
		&= \sum_{\theta} \mathrm{r}_{\theta/\lambda}\, \rho(s_{\theta})
\end{align*}
as needed.
\end{proof}

\section{Grothendieck-positive specializations}\label{gps}
\begin{definition}
A homomorphism 
$\varphi : \Gamma \to \mathbb{R}$ is called {\it $G$-positive} if $\tilde G_{\lambda}(\varphi) \ge 0$ for all $\lambda$.
\end{definition}
Define the following positive cone in $\Gamma$ of nonnegative linear combinations of $\{\tilde G^{}_{\lambda} \}$: 
$$\Gamma_+ := \bigoplus_{\lambda} \mathbb{R}_{\ge 0} \cdot \tilde G^{}_{\lambda}\ \subset\ \Gamma$$ 
Note that it is 
closed under multiplication since products of symmetric Grothendieck polynomials expand positively 
in $\{\tilde G^{}_{\lambda} \}$. 
\begin{lemma}
Let $\varphi$ be a $G$-positive homomorphism of $\Gamma$. Then $\tilde G_{\lambda/\!\!/\mu}(\varphi) \ge 0$.
\end{lemma}
\begin{proof}
Follows from the fact $\tilde G_{\lambda/\!\!/\mu} \in \Gamma_+$ is Grothendieck-positive, see \eqref{pprod}.
\end{proof}

\begin{lemma}
Let $\varphi_1, \varphi_2$ be $G$-positive homomorphisms of $\Gamma$. Then the union $\varphi = (\varphi_1, \varphi_2)$ is also $G$-positive.
\end{lemma}
\begin{proof}
Immediate from the definition~\ref{uu} of union and previous lemma. 
\end{proof}

Consider the elements $H^{}_{n} \in \Gamma_+$ defined as follows
$$H_0^{} := 1 + \tilde G^{}_{(1)}, \quad H_n^{} := \tilde G^{}_{(n)} + \tilde G^{}_{(n + 1)}\ 
\text{ for } n > 0, \text{ and } H^{}_{n} = 0 \text{ for } n < 0.$$

For any partition $\mu$, define the function
\begin{equation}\label{flam}
F_{\mu} := \det\left[H^{}_{\mu_i - i + j}\right]_{1 \le i, j \le \ell(\mu)} \in \Gamma
\end{equation}
To describe expansion of $F_{\mu}$ in the basis $\{\tilde G_{\lambda} \}$, we need to define the following types of tableaux.
\begin{definition}
A {\it delegant tableau} of shape $\lambda/\mu$ is a filling of the boxes of $\lambda/\mu$ with positive integers such that 
\begin{itemize}
\item the entries weakly decrease in rows from left to right and strictly decrease from top to bottom
\item  the elements in row $i$ lie in $[\lambda_i,\ \ell + 1 + \mu_i - i]$, where $\ell = \ell(\mu)$.
\end{itemize}
Let $\mathrm{d}_{\lambda/\mu}$ be the number of delegant tableaux of shape $\lambda/\mu$. Note that from the second property we must have $\lambda_1 \le \ell + \mu_1$ and $\ell(\lambda) =\ell= \ell(\mu)$. Note also that for any fixed $\mu$ there is only finitely many $\lambda$ such that $\mathrm{d}_{\lambda/\mu} > 0$.
\end{definition}

The core result that allows us to describe $G$-positive specializations is the following theorem showing that the function $F_{\mu}$ is Grothendieck-positive.

\begin{theorem}\label{fgplus}
We have the following expansion
$$
F_{\mu} = \sum_{\nu} 
\mathrm{d}_{\nu/\mu}\, \tilde G^{}_{\nu},
$$
where $\mathrm{d}_{\nu/\mu}$ is the number of delegant tableaux of shape $\nu/\mu$. In particular, $F_{\mu} \in \Gamma_+$. 
\end{theorem}

An application of this theorem is the following.
\begin{theorem}\label{toep}
Let $\varphi$ be a $G$-positive homomorphism of $\Gamma$. Then the Toeplitz matrix 
$$H(\varphi) := \left[H^{}_{i - j}(\varphi)\right]_{i,j \ge 0}$$ 
is totally nonnegative.
\end{theorem}
\begin{proof}
First note that the diagonal elements $H_0(\varphi) = 1 + \tilde G_{1}(\varphi) > 0$ are positive and since the matrix $H(\varphi)$ is upper triangular, the principal minors $[H^{}_{i - j}(\varphi)]_{0 \le i,j \le n}$ are positive for all $n \ge 0$. For all partitions $\lambda,$ the minors $\det\left[H^{}_{\lambda_i - i + j}(\varphi)\right]  = F_{\lambda}(\varphi) \ge 0$ are nonnegative by Theorem~\ref{fgplus}. In particular, all minors occupying several consecutive columns are nonnegative. Nonnegativity of all these minors and positivity of the principal minors is sufficient for the matrix $H$ to be totally nonnegative (see the nonnegativity criterion \cite[Thm.~3.2]{gp}; see also \cite{fz} for total positivity tests). 
\end{proof}

To prove Theorem \ref{fgplus} we need some preparatory lemmas.

\begin{lemma}\label{lhb}
We have 
$H_n^{} = H_{0}^{} \cdot h_n$ for all $n \ge 0$. Equivalently, we have the generating series 
\begin{equation}
 \sum_{n = 0}^{\infty} H_n\, z^n = 1 + (z + 1)\sum_{n = 1}^{\infty} \tilde G_{(n)}\, z^{n-1} =  \prod_{n = 1}^{\infty} \frac{1 + x_n}{1 - z x_n}
\end{equation}
\end{lemma}
\begin{proof}
It is known that (see \cite{lenart})
$$
\tilde G_{(n)} = \sum_{i = 0}^{\infty} s_{(n-1 | i)},
$$
where $s_{(a | b)} = s_{(a + 1, 1^b)}$ is the Schur function of hook shape. On the other hand, using the Pieri rule for Schur functions we get 
$$
H_{0}^{} \cdot  h_n = (1 + \tilde G_{(1)}) \cdot h_n = \sum_{i = 0}^{\infty} e_i\, h_n = \sum_{i = 0}^{\infty} s_{(n-1 | i)} + s_{(n | i)} = \tilde G_{(n)} + \tilde G_{(n+1)} = H_n.
$$
The generating function identity follows as well.
\end{proof}

\begin{lemma}\label{ddet} 
The following determinantal formula holds
$$
\mathrm{d}_{\nu/\mu} = \det\left[ \binom{\ell - i + 1}{\nu_i - i - \mu_j + j}\right]_{1 \le i,j \le \ell}, 
$$
where $\ell = \ell(\mu) = \ell(\nu)$.
\end{lemma}
\begin{proof}
A standard exercise on the Lindstr\"om-Gessel-Viennot Lemma, see e.g. \cite[Thm.~15]{gv}. 
\end{proof}

\begin{lemma}\label{fdet}
The following determinantal formula holds
$$
\mathrm{f}_{\nu/\lambda} = \det\left[ \binom{\nu_i - \lambda_j + j - 2}{\nu_i - i - \lambda_j + j}\right]_{1\le i,j \le \ell(\nu)},
$$
where $\mathrm{f}_{\nu/\lambda}$ is the number of elegant tableaux of shape $\nu/\lambda$ (see Def.~\ref{eleg}).
\end{lemma}
\begin{proof}
Follows from the lattice path interpretation given in \cite{lenart} (cf. \cite{dy}) and the Lindstr\"om-Gessel-Viennot Lemma \cite{gv}.
\end{proof}

\begin{proof}[Proof of Theorem \ref{fgplus}]
By definition of $F_{\mu}$ it is clear that $F_{\mu} \in \Gamma$. Moreover, the expansion of $F_{\mu}$ in the basis $\{\tilde G^{}_{\nu} \}$ contains only elements with $\ell(\nu) \le \ell(\mu)$. To see this notice that since the Pieri rule \eqref{gpieri0} adds horizontal strips, we have that products $\tilde G_{(n_1)} \cdots \tilde G_{(n_\ell)}$ contains only elements $\tilde G_{\nu}$ with $\ell(\nu) \le \ell$.
Let $\ell = \ell(\mu)$. Lemma \ref{lhb} gives 
$$
F_{\mu} = \det\left[H^{}_{0} \cdot h_{\mu_i - i + j}\right]_{1 \le i, j \le \ell}  = (H^{}_{0})^{\ell}\, \det\left[h_{\mu_i - i + j}\right]_{1 \le i, j \le \ell} = (H_{0}^{})^{\ell}\, s_{\mu}. 
$$
Using the following identity for Schur functions (which can be derived easily from Cauchy identities and LR expansions, or follows from skew Cauchy identities, see e.g. \cite[Ch.~1]{macdonald}) 
\begin{equation*}
s_{\mu}(\mathbf{x}) \prod_{i,j = 1}^{\infty} (1 + x_i y_j) = \sum_{\lambda \supset \mu} s_{\lambda}(\mathbf{x}) s_{\lambda'/\mu'}(\mathbf{y})
\end{equation*}
for the specialization $\mathbf{y} = (1^{\ell})$ and the Grothendieck expansion \eqref{gschur} 
we obtain
\begin{align*}
(H_{0}^{})^{\ell}\, s_{\mu}(\mathbf{x}) 
	&= s_{\mu}(\mathbf{x})  \prod_{i = 1}^{\infty} (1 + x_i)^\ell \\
	&= \sum_{\lambda \supset \mu} s_{\lambda}(\mathbf{x})\, s_{\lambda'/\mu'}(1^\ell) 
	\\
	&= \sum_{\nu} \tilde G^{}_{\nu}(\mathbf{x}) 
	\sum_{\mu \subset \lambda \subset \nu} (-1)^{|\nu/\lambda|} \mathrm{f}_{\nu/\lambda}\, s_{\lambda'/\mu'}(1^\ell)
\end{align*}
where the sum is finite and contains only elements $\tilde G_{\nu}$ with $\ell(\nu) \le \ell$. Let 
$$
d(\nu,\mu) := \sum_{\mu \subset \lambda \subset \nu} (-1)^{|\nu/\lambda|} \mathrm{f}_{\nu/\lambda}\, s_{\lambda'/\mu'}(1^\ell)
$$
and we need to show that $d(\nu,\mu) = \mathrm{d}_{\nu/\mu}$ is the number of delegant tableaux. 

From the Jacobi-Trudi identity we have (note that $\ell(\lambda) \le \ell$)
$$
s_{\lambda'/\mu'}(1^\ell) =\det\left[ e_{\lambda_i - i - \mu_j + j}(1^{\ell}) \right]_{1\le i,j \le \ell} = \det\left[ \binom{\ell}{\lambda_i - i - \mu_j + j}\right]_{1\le i,j \le \ell}
$$
Combining this determinantal formula, the formula in Lemma~\ref{fdet}, the Cauchy-Binet identity (in such form indexed by partitions can be found in \cite{dy}), and the formula in Lemma~\ref{ddet} we obtain
\begin{align*}
d(\nu, \mu) 
	&=  \sum_{\mu \subset \lambda \subset \nu} (-1)^{|\nu/\lambda|} \det\left[ \binom{\nu_i - \lambda_j + j - 2}{\nu_i - i - \lambda_j + j}\right]  \det\left[ \binom{\ell}{\lambda_i - i - \mu_j + j}\right] \\
	&=  \sum_{\mu \subset \lambda \subset \nu} \det\left[ (-1)^{\nu_i - i - \lambda_j + j}\binom{\nu_i - \lambda_j + j - 2}{\nu_i - i - \lambda_j + j}\right]  \det\left[ \binom{\ell}{\lambda_i - i - \mu_j + j}\right] \\
	&=  \det\left[\sum_{k} (-1)^{\nu_i - i - k}\binom{\nu_i - k - 2}{\nu_i - i - k} \binom{\ell}{k - \mu_j + j}\right] \\
	&=  \det\left[ \binom{\ell - i + 1}{\nu_i - i - \mu_j + j}\right]_{1 \le i,j \le \ell}\\ 
	&= \mathrm{d}_{\nu/\mu}
\end{align*}
Here we also used the identity 
$$
\sum_{k} (-1)^{\nu_i - i - k}\binom{\nu_i - k - 2}{\nu_i - i - k} \binom{\ell}{k - \mu_j + j} = \binom{\ell - i + 1}{\nu_i - i - \mu_j + j}
$$
which can be derived by comparing the coefficients at $x^{\nu_i - i - \mu_j + j}$ from both sides of the generating function identity 
$(1 + x)^{\ell-i+1}  = (1 + x)^\ell (1 + x)^{-(i-1)}.$
\end{proof}

\begin{proof}[Proof of Theorem \ref{t1}]
By Theorem \ref{toep} and the Edrei-Thoma theorem we have that the map
$$
h_n \longmapsto \frac{H_n(\varphi)}{H_{0}(\varphi)} = \frac{\tilde G_{(n)}(\varphi) + \tilde G_{(n+1)}(\varphi)}{1 + \tilde G_{(1)}(\varphi)}, \quad n = 1, 2, \ldots
$$
defines a Schur-positive specialization of $\Lambda$. 
\end{proof}

\begin{lemma}\label{zero}
Let $\varphi$ be a $G$-positive homomorphism of $\Gamma$. Suppose $\tilde G_{\mu}(\varphi) = 0$ for some $\mu$. Then $\tilde G_{\lambda}(\varphi) = 0$ for all $\lambda \supset \mu$.
\end{lemma}
\begin{proof}
From the Pieri rule we have
$$0 = \varphi(\tilde G_{(1)} \cdot \tilde G_{\mu}) = \sum_{\lambda/\mu \text{ rook strip}} \tilde G_{\lambda}(\varphi) \ge 0$$ 
Hence $\tilde G_{\lambda}(\varphi) = 0$ for all $\lambda$ such that $\lambda/\mu$ is a rook strip. In particular, $\tilde G_{\lambda}(\varphi) = 0$ for all $\lambda = \mu + \square$. Applying the same argument replacing $\mu$ with these partitions $\lambda$ further, we obtain that $\tilde G_{\lambda}(\varphi) = 0$ for all $\lambda \supset \mu$. 
\end{proof}

\begin{proof}[Proof of Theorem \ref{t2}]
Since $\tilde G_{(1)}(\varphi) = 1$, the map 
$$
\rho: h_n \longmapsto \frac{1}{2}\left(\tilde G^{}_{(n)}(\varphi) + \tilde G^{}_{(n+1)}(\varphi)\right)
$$
defines a Schur-positive specialization by Theorem \ref{t1}. Therefore, 
$$
\rho(H(z)) = 1 + \frac{1}{2}\sum_{n}\left(\tilde G^{}_{(n)}(\varphi) + \tilde G^{}_{(n+1)}(\varphi)\right) z^n  = 1 + \sum_{n} \rho(h_n) z^n = e^{\gamma z} \prod_{n} \frac{1 + \beta_n z}{1 - \alpha_n z}
$$
for some nonnegative reals $\{\alpha_{n} \}$, $\{ \beta_n\},$ $\gamma$ such that $\sum_n (\alpha_n + \beta_n) < \infty$. 

Note that from the Pieri rule 
\begin{equation}\label{gpieri}
\tilde G_{(n)} \cdot \tilde G_{(1)} = \tilde G_{(n+1)} + \tilde G_{(n,1)} + \tilde G_{(n+1,1)}
\end{equation}
for all $n$ we obtain
$$\tilde G_{(n+1)}(\varphi) \le \tilde G_{(1)}(\varphi) \cdot \tilde G_{(n)}(\varphi) = \tilde G_{(n)}(\varphi) \le \tilde G_{(1)}(\varphi) = 1.$$

Consider two cases:

{\it Case 1.} If $\tilde G_{(2)}(\varphi) = 1$. Then the identity 
$(\tilde G_{(1)})^2 = \tilde G_{(2)} + \tilde G_{(1^2)} + \tilde G_{(2,1)}$ implies that $\tilde G_{(1^2)}(\varphi) = 0$ and hence by Lemma~\ref{zero} we have $\tilde G_{\lambda}(\varphi) = 0$ for all $\lambda \supset (1^2)$ (i.e. when $\lambda$ has at least two rows). From the identity \eqref{gpieri} 
we obtain by induction that $\tilde G_{(n)}(\varphi) = 1$ for all $n \ge 1$. Therefore, 
$\rho(H(z)) = \frac{1}{1 - z}$ which corresponds to the parameters $\gamma = 0$, $\alpha_1 = 1$, $\alpha_{n+1} = 0$, $\beta_n = 0$ for all $n$. Conversely, we have the $G$-positive homomorphism $\tilde G_{\lambda}(\varphi) = \tilde G_{\lambda}(1)$ that meets these conditions.

{\it Case 2.} If $\tilde G_{(2)}(\varphi) = \delta < 1$. 
From nonnegativity of $2 \times 2$ minors of the matrix $[\rho(h_{i - j})]_{i,j \ge 0}$ (which is log-concavity) we have  
$$
\frac{\rho(h_{n})}{\rho(h_{n+1})} 
\ge \frac{\rho(h_{n-1})}{\rho(h_{n})} \ge \ldots \ge \frac{1}{\rho(h_1)} = \frac{2}{1 + \tilde G_{(2)}(\varphi)} = \frac{2}{1 + \delta} >  1.
$$
Therefore, the radius of convergence of $\rho(H(z))$ is greater than $1$. (In particular, $\alpha_n < 1$.) 
Let now 
$$
\rho(E(z)) := 1 + \sum_{n = 1}^{\infty} \rho(e_n)\, z^n = \rho(H(-z)^{-1})= e^{\gamma z} \prod_{n} \frac{1 + \alpha_n z}{1 - \beta_n z}.
$$
Similarly, we obtain that
$$
\frac{\rho(e_{n})}{\rho(e_{n+1})} 
\ge \frac{1}{\rho(e_1)} = \frac{1}{\rho(h_1)} = \frac{2}{1 + \tilde G_{(2)}(\varphi)} = \frac{2}{1 + \delta} >  1
$$
and hence the radius of convergence of $\rho(E(z))$ is greater than $1$ as well. In particular, $\beta_n < 1$ for all $n$.
Now we have 
$$
\rho(H(-1)) = \frac{1}{2} =  e^{-\gamma} \prod_{n} \frac{1 - \beta_n}{1 + \alpha_n} < \infty
$$
and hence 
$$
\gamma = \log 2 - \sum_{n} \log(1 + \alpha_n) + \sum_{n} \log(1 - \beta_n).
$$

Let us show the converse part. Suppose $\rho$ is a Schur-positive specialization of $\Lambda$ with the parameters as above. Let us show that it extends to a $G$-positive specialization $\hat\rho$ such that 
$$\tilde G_{\lambda}(\hat\rho) := \sum_{\mu \supset \lambda} \mathrm{r}_{\mu/\lambda}\, \rho(s_{\mu}) \ge 0,$$
relying on Schur expansion of $\tilde G_{\lambda}$, see \eqref{gschur}.
We are going to show that the infinite sums
$$
0 \le \sum_{\mu \supset \lambda} \mathrm{r}_{\mu/\lambda}\, \rho(s_{\mu}) < \infty
$$
converge. 
Let us check the normalization: 
$$\rho(H(-1)) = 1 - \rho(h_1) + \rho(h_2) - \ldots = e^{-\gamma} \prod_{n} \frac{1 - \beta_n}{1 + \alpha_n} = \frac{1}{2}$$
Hence,
$$
\tilde G_{(1)}(\hat\rho) = -1 + \frac{1}{\rho(H(-1))} = 1. 
$$
Let $\lambda \vdash n$ and consider the expansion
\begin{equation}\label{gek}
(\tilde G_{(1)})^n - \tilde G_{\lambda} = \sum_{\nu} a_{\nu}\, \tilde G_{\nu} \in \Gamma_+
\end{equation}
Let also take the Schur expansion (in $\hat\Lambda$)
$$
(\tilde G_{(1)})^n = \sum_{\nu} b_{\nu}\, s_{\nu}, \quad b_{\nu} \in \mathbb{Z}_{\ge 0}
$$
Applying the specialization $\rho$ to these series we have
$$
 \sum_{\nu} b_{\nu}\, \rho(s_{\nu}) = 1
$$
converges. 
Finally observe that from \eqref{gek}
$$
0 \le \sum_{\mu \supset \lambda} \mathrm{r}_{\mu/\lambda}\, \rho(s_{\mu}) \le \sum_{\nu} b_{\nu}\, \rho(s_{\nu}) = 1.
$$
\end{proof}

\subsection{Structure of $G$-positive specializations}
Let us show that Schur-positive generators $\phi, \varepsilon, \pi$ (see subsec.~\ref{sgen}) extend to $G$-positive specializations $\hat\phi, \hat\varepsilon, \hat\pi$.

\begin{lemma}\label{sext}
Let $\alpha, \beta, \gamma \in \mathbb{R}_{\ge 0}$ and $\beta < 1$. Then the Schur-positive generators $\phi_{\alpha}, \varepsilon_{\beta}, \pi_{\gamma}$ of $\Lambda$ extend to $G$-positive specializations $\hat\phi_{\alpha}, \hat\varepsilon_{\beta}, \hat\pi_{\gamma}$ of $\Gamma$.
\end{lemma}
\begin{proof}
Let us check that 
$$\tilde G_{\lambda}(\hat\phi_{\alpha}) =  
\begin{cases}
	\alpha^{|\lambda|}, & \text{ if } \ell(\lambda) = 1,\\
	0, & \text{ otherwise}. 
\end{cases}
$$
Indeed, 
$$\sum_{\mu} \mathrm{r}_{\mu/(n)}\, \phi_{\alpha}(s_{\mu}) = \phi_{\alpha}(s_{(n)}) = \alpha^n$$ as $\mathrm{r}_{\mu/(n)} = 0$ if $\mu \supset (n)$ and $\mu \ne (n)$ and $\sum_{\mu \supset \lambda} \mathrm{r}_{\mu/\lambda}\, \phi_{\alpha}(s_{\mu}) = 0$ if $\ell(\lambda) \ge 2$ since  $\ell(\mu) \ge 2$ when $\mathrm{r}_{\mu/\lambda} > 0$ and then $\phi_{\alpha}(s_{\mu}) = 0$. 
Equivalently, $\hat\phi_{\alpha}$ just corresponds to the specialization $(x_1, x_2, \ldots) \mapsto (\alpha, 0, 0, \ldots)$.

To see the effect of $\hat\varepsilon_{\beta}$, note that the involution $\omega : \hat\Lambda \to \hat\Lambda$ (extended on the completion $\hat\Lambda$ via $\omega(s_{\lambda}) = s_{\lambda'}$) satisfies (see \cite{pp1, dy}) 
$$\omega(\tilde G_{\lambda}(x_1, x_2, \ldots)) = \tilde G_{\lambda'}\left(\frac{x_1}{1 - x_1}, \frac{x_2}{1 - x_2}, \ldots\right).$$ 
Since $\varepsilon_{\beta} = \phi_{\beta} \circ \omega$ and $\hat\varepsilon_{\beta} = \hat\phi_{\beta} \circ \omega$ we obtain
$$
\tilde G_{\lambda}(\hat\varepsilon_{\beta}) = 
\begin{cases}
	\left(\frac{\beta}{1 - \beta} \right)^{|\lambda|}, & \text{ if } \lambda_1 = 1,\\
	0, & \text{ otherwise}. 
\end{cases}
$$

As for the Plancherel specialization $\hat\pi_\gamma$ we have 
$$
\sum_{\mu} \mathrm{r}_{\mu/\lambda}\, \pi_{\gamma}(s_{\mu}) 
= \sum_{\mu} \mathrm{r}_{\mu/\lambda}\, \frac{\gamma^{|\mu|} f^{\mu}}{|\mu|!} \in \mathbb{R}[[\gamma]]
$$
Let us show that this series converges for all $\gamma \in \mathbb{R}_{\ge 0}$ (and hence for $\gamma \in \mathbb{R}$). Note that for $\lambda = (1)$ we have:
$$
\sum_{\mu \supset (1)} \mathrm{r}_{\mu/(1)}\, \pi_{\gamma}(s_{\mu}) 
= 
\sum_{n \ge 1} \mathrm{r}_{(1^n)/(1)}\, \pi_{\gamma}(s_{(1^n)})  = \sum_{n = 1}^{\infty} \frac{\gamma^n}{n!} = -1 + e^{\gamma},
$$
which gives the value of $\tilde G_{(1)}(\hat\pi_{\gamma})$.
Consider the expansion 
$$
(\tilde G_{(1)})^{|\lambda|} = \sum_{\nu} a_{\nu}\, \tilde G_{\nu} = \sum_{\mu} b_{\mu}\, s_{\mu}, \quad a_{\nu}, b_{\mu} \ge 0.
$$
Note that the function $(\tilde G_{(1)})^{|\lambda|} - \tilde G_{\lambda}$ is both Grothendieck and Schur positive and hence we have $b_{\mu} \ge \mathrm{r}_{\mu/\lambda}$ implying
$$
0 \le \sum_{\mu} \mathrm{r}_{\mu/\lambda}\, \pi_{\gamma}(s_{\mu})  \le \sum_{\mu} b_{\mu}\, \pi_{\gamma}(s_{\mu}) = (-1 + e^{\gamma})^{|\lambda|}.
$$
Therefore, the series 
$$
\tilde G_{\lambda}(\hat\pi_{\gamma}) = \sum_{\mu} \mathrm{r}_{\mu/\lambda}\, \pi_{\gamma}(s_{\mu})
$$
converges for all $\gamma \in \mathbb{R}_{\ge 0}$ and the homomorphism $\hat\pi_{\gamma}$ is a well-defined and $G$-positive.
\end{proof}

\begin{proposition}\label{phig}
Let $\varphi$ be the union of extended Schur-positive generators $(\hat\phi_{\alpha_n}),$ $(\hat\varepsilon_{\beta_n})$, $\hat\pi_{\gamma}$ 
for nonnegative real parameters
$\{\alpha_n\}$,  $\{\beta_n < 1\}$, $\gamma$, such that  $\sum_{n} (\alpha_n + \beta_n) < \infty$. 
Then $\varphi$ is $G$-positive. 
In particular, we have 
\begin{equation}\label{g1p}
\tilde G_{(1)}(\varphi) = -1 + e^{\gamma} \prod_{n = 1}^{\infty} \frac{1 + \alpha_n}{1 - \beta_n}
\end{equation}
and the following generating function for the elements $\{ \tilde G_{(n)}(\varphi)\}$
\begin{equation}\label{eq3}
1 + (z + 1) \sum_{n = 1}^{\infty} \tilde G^{}_{(n)}(\varphi) z^{n - 1} = e^{\gamma (z + 1)} \prod_{n = 1}^{\infty} \frac{1 + \alpha_n}{1 - \alpha_n z} \prod_{n = 1}^{\infty} \frac{1 + \beta_n z}{1 - \beta_n}
\end{equation}
\end{proposition}
\begin{proof}
By Lemma~\ref{sext}, the specializations $\hat\phi_{\alpha_n},$ $\hat\varepsilon_{\beta_n}$, $\hat\pi_{\gamma}$ are all $G$-positive and hence their union is also $G$-positive given $\tilde G_{(1)}(\varphi)$ converges. The formulas \eqref{g1p}, \eqref{eq3} then follow from Schur expansions of $\{\tilde G_{(n)}\}$, see Lemma~\ref{lhb}.
\end{proof}

\begin{remark}
Instead of the $G$-positive generators $\hat\varepsilon_{\beta'}$ for $\beta' < 1$ we can use another specializations $\psi_{\beta} := \hat\phi_{\beta} \circ \tau$, i.e. $\psi_{\beta} : \tilde G_{\lambda} \mapsto \tilde G_{\lambda'}(\hat\phi_{\beta})$ for $\beta \ge 0$. Here we have $\tilde G_{\lambda}(\psi_{\beta}) = \beta^{|\lambda|} \ge 0,$ if $\lambda_1 = 1$ and $0$ otherwise. One can see that $\psi_{\beta} = \hat\varepsilon_{\beta'}$, where $\beta' = \beta/(1 + \beta)$. Then we can alternatively redefine $G$-positive homomorphisms $\varphi$ in Theorem~\ref{phig} as the union of generators $\hat\phi_{\alpha_n},$ $\psi_{\beta_n},$ and $\hat\pi_{\gamma}$ for nonnegative reals $\{ \alpha_n\}$, $\{\beta_n \},$ $\gamma$ such that $\sum_{n}(\alpha_n + \beta_n) < \infty$. 
In particular, 
$$
\tilde G_{(1)}(\varphi) = -1 + e^{\gamma} \prod_{n = 1}^{\infty} {(1 + \alpha_n)(1 + \beta_n)}.
$$
\end{remark}

\section{Harmonic functions  and boundary of a filtered Young's graph}
Consider the (infinite) {\it filtered Young graph} $\widetilde{\mathbb{Y}}^{}$ defined as follows:
\begin{itemize}
\item[(i)] its vertices are labeled by partitions $\lambda$;
\item[(ii)] there is an arc $\lambda \to \mu$ iff $\mu/\lambda$ is a {\it rook strip} (i.e. no two boxes lie in the same row or column)
\end{itemize}

\begin{definition}
Say that a function $\varphi : \mathcal{P} \to \mathbb{R}_{\ge 0}$ is {\it harmonic} on $\widetilde{\mathbb{Y}}$ if 
$$
\varphi(\varnothing) = 1\quad \text{ and }\quad \varphi(\lambda) = \sum_{\lambda \to \mu} 
\varphi(\mu).
$$
\end{definition}
Let $H^{}(\widetilde{\mathbb{Y}})$ be the set of harmonic functions on $\widetilde{\mathbb{Y}}$. Then $H^{}(\widetilde{\mathbb{Y}})$ is a convex set and let $\partial \widetilde{\mathbb{Y}}$ be the 
set of {\it extreme points} of $H^{}(\widetilde{\mathbb{Y}})$, or the {\it boundary} of $\widetilde{\mathbb{Y}}$. 

For every function $\varphi \in H^{}(\widetilde{\mathbb{Y}})$ define the linear functional 
$\hat\varphi : \Gamma \to \mathbb{R}$ such that 
$$\hat\varphi(\tilde G_{\lambda}) := \varphi(\lambda).$$
Alternatively, consider linear functionals $\hat\varphi : \Gamma \to \mathbb{R}$ satisfying the following properties:
\begin{itemize}
\item[(i)] $\hat\varphi(1) = 1$ 

\item[(ii)] $\hat\varphi(\tilde G^{}_{\lambda}) \ge 0$ 

\item[(iii)] $\hat\varphi(\tilde G^{}_{(1)} \cdot \tilde G_{\lambda}) = \hat\varphi(\tilde G_{\lambda})$ 
\end{itemize}
Using the simple {Pieri rule} we obtain
$$
\hat\varphi(\tilde G^{}_{\lambda}) = \hat\varphi(\tilde G^{}_{(1)} \cdot \tilde G^{}_{\lambda}) = \sum_{\mu/\lambda \text{ rook strip}} 
\hat\varphi(\tilde G^{}_{\mu})
$$
i.e. the function $\varphi(\lambda) = \hat\varphi(\tilde G_{\lambda})$ is harmonic on $\widetilde{\mathbb{Y}}$.

The next theorem is a version of the Vershik--Kerov ``ring theorem" \cite{vk} for $\widetilde{\mathbb{Y}}$.
\begin{theorem}
We have: $\varphi \in \partial \widetilde{\mathbb{Y}}$ i.e. $\varphi$ is extreme if and only if the linear functional $\hat\varphi$ is a normalized $G$-positive homomorphism of $\Gamma$. 
\end{theorem}

\begin{proof}
The same proof as in the case of graded Young graph and the ring $\Lambda$ with the Schur basis (see \cite{bool}, also \cite{go} for a general statement) works in our case as $\tilde G_{\mu}\, \tilde G_{\nu} \in \Gamma_+$ and $(\tilde G_{(1)})^{|\lambda|} - \tilde G_{\lambda} \in \Gamma_+$ for all $\lambda, \mu, \nu$. We reproduce it here for completeness.

Suppose $f \in \Gamma_+$ such that $\hat\varphi(f) > 0$. Let us then check that the new function 
$$\varphi_f(g) := \frac{\hat\varphi(f g)}{\hat\varphi(f)}, \quad g \in \Gamma$$ 
also satisfies the properties (i)--(iii) of linear functionals $\hat\varphi$. Indeed, we have (i) $\varphi_f(1) = 1$, (ii) $\varphi_f(\tilde G_{\lambda}) = {\hat\varphi(f \cdot \tilde G_{\lambda})}/{\hat\varphi(f)} \ge 0$ since $f \cdot \tilde G_{\lambda} \in \Gamma_+$ expands as a nonnegative linear combination of $\{\tilde G_{\mu} \}$, and for (iii) we have 
$$\varphi_f(\tilde G_{(1)} \cdot \tilde G_{\lambda}) = {\hat\varphi(f \cdot \tilde G_{(1)} \cdot \tilde G_{\lambda})} / {\hat\varphi(f)}= {\hat\varphi(f \cdot \tilde G_{\lambda})} / {\hat\varphi(f)} = \varphi_f(\tilde G_{\lambda}).$$

Suppose that $\varphi$ is extreme. We need to show that 
$\hat\varphi(\tilde G^{}_{\mu} \cdot G) = \hat\varphi(\tilde G^{}_{\mu})\, \hat\varphi(G)$ for all $\mu \in \mathcal{P}$ and $G \in \Gamma$. 
Note that for $\lambda \vdash n$ we have $(\tilde G_{(1)})^n - \tilde G_{\lambda} \in \Gamma_+$ as a consequence of Pieri rule.

If $\hat\varphi(\tilde G^{}_{\mu}) = 0,$ then for $\lambda \vdash n$ we have
$$
0 \le \hat\varphi(\tilde G^{}_{\mu} \cdot \tilde G^{}_{\lambda}) \le \hat\varphi(\tilde G^{}_{\mu}\cdot  (\tilde G^{}_{(1)})^n) = \hat\varphi(\tilde G_{\mu}) = 0.
$$
Hence $\hat\varphi(\tilde G^{}_{\mu} \cdot \tilde G^{}_{\lambda}) = 0$ and the multiplicativity holds.

If $\hat\varphi(\tilde G^{}_{\mu}) > 0$ for $\mu \vdash m$, consider the functions
$$
g_1 := \frac{1}{2}\, \tilde G^{}_{\mu}, \qquad g_2 := (\tilde G^{}_{(1)})^m - g_1 \in \Gamma_+.
$$
Then for $G \in \Gamma$ we have
$$
\hat\varphi(G) = \hat\varphi((\tilde G^{}_{(1)})^{m} \cdot G) = \hat\varphi(g_1 \cdot G) + \hat\varphi(g_2 \cdot G).
$$
We clearly have $c_1 := \hat\varphi(g_1) > 0$, $c_2 := \hat\varphi(g_2) > 0$ such that  $c_1 + c_2 = 1$ 
and hence
$$
\hat\varphi = c_1\, \varphi_{g_1} + c_2\, \varphi_{g_2}.
$$
Since $\varphi$ is extreme we have $\hat\varphi = \varphi_{g_1}$ which implies that $\hat\varphi(G) = \hat\varphi(\tilde G_{\mu} \cdot G) / \hat\varphi(\tilde G_{\mu})$ as needed.

To show the converse part, suppose that $\hat\varphi$ is a homomorphism of $\Gamma$, and let us show that $\varphi$ is extreme. By Choquet's theorem we have the integral presentation
$$
\hat\varphi(G) = \int_{f \in \partial \widetilde{\mathbb{Y}}} \hat f(G)\, \mu(df), \quad G \in \Gamma
$$
where $\mu$ is a probability measure on $\partial \widetilde{\mathbb{Y}}$. View the function $f \to \hat f(G)$ as a random variable on the space $(\partial \widetilde{\mathbb{Y}}, \mu)$. By the integral representation its expectation is $\hat\varphi(G)$. We also have $\hat f$ is a homomorphism of $\Gamma$ by the argument above. 
Now observe that
$$
\left(\int_{f \in \partial \widetilde{\mathbb{Y}}} \hat f(G)\, \mu(df) \right)^2 = \hat\varphi(G)^2 = \hat\varphi(G^2) 
= \int_{f \in \partial \widetilde{\mathbb{Y}}} \hat f(G^2)\, \mu(df) = \int_{f \in \partial \widetilde{\mathbb{Y}}} \hat f(G)^2 \mu(df)
$$
implying that 
$\hat f(G)$  
has variance $0$. 
Therefore, $\hat f(G) = \hat\varphi(G)$ almost surely.
Hence $\mu$ is a delta-measure and $\varphi$ is extreme.
\end{proof}

\section{Positivity of Grothendieck polynomials with alternating signs}\label{galt} 
Let us define the functions 
$${G}_{\lambda/\!\!/\mu}(\mathbf{x}) := (-1)^{|\lambda/\mu|}\, \tilde G_{\lambda/\!\!/\mu}(-\mathbf{x})$$ which is a usual definition of symmetric Grothendieck polynomials whose monomial expansion has alternating signs. Namely, we have (see \cite{dy2})
$$
{G}_{\lambda/\!\!/\mu}(\mathbf{x}) = \sum_{T \in SSVT(\lambda/\!\!/\mu)} (-1)^{|T| - |\lambda/\mu|} x^{T}, 
$$
where $|T|$ is the number of entries in $T$.
Clearly, $\{{G}_{\lambda}\}$ is a basis of $\Gamma$ and we have 
$$
{G}_{\mu} \cdot {G}_{\nu} = \sum_{\lambda} (-1)^{|\lambda| - |\mu| - |\nu|} c^{\lambda}_{\mu \nu}\, {G}_{\lambda}.
$$ 
Note also that 
$$
{G}_{(1)/\!\!/(1)} = 1 - {G}_{(1)} = \prod_{n = 1}^{\infty} (1 - x_n).
$$ 
\begin{definition}
We say that a homomorphism $\varphi : \Gamma \to \mathbb{R}$ is {\it $\overline{G}$-positive} if $\varphi({G}_{\lambda/\!\!/\mu}) \ge 0$ for all $\lambda, \mu$.
\end{definition}

\begin{lemma}
Let $\varphi_1, \varphi_2$ be $\overline{G}$-positive homomorphisms of $\Gamma$. Then their union $\varphi = (\varphi_1, \varphi_2)$ is also $\overline{G}$-positive.
\end{lemma}
\begin{proof}
Follows from the branching formula \eqref{br} which remains the same for the functions ${G}$.
\end{proof}

\begin{lemma}\label{gine}
Let $\varphi$ be a $\overline{G}$-positive homomorphism of $\Gamma$. Then 
\begin{equation*}
1 \ge {G}_{(1)}(\varphi) \ge {G}_{(2)}(\varphi) \ge \ldots \ge 0. 
\end{equation*}
\end{lemma}
\begin{proof}
We have for all $n \ge 0$
$$
{G}_{(n)}(\varphi) - {G}_{(n+1)}(\varphi) = {G}_{(n+1)/\!\!/(1)}(\varphi) \ge 0. 
$$
Note also that ${G}_{(0)}(\varphi) = {G}_{\varnothing}(\varphi) = 1$.
\end{proof}

\begin{theorem}
Let $\varphi$ be a $\overline{G}$-positive homomorphism of $\Gamma$ and suppose ${G}_{(1)}(\varphi) \ne 1$. Then  
$$
\rho : h_n \longmapsto \frac{{G}_{(n)}(\varphi) - {G}_{(n+1)}(\varphi) }{1 - {G}_{(1)}(\varphi) }
$$
defines a Schur-positive specialization of $\Lambda$ with nonnegative parameters $\{\alpha_{n}\}$, $\{\beta_n \}$, $\gamma$ such that 
\begin{equation}\label{gcond}
{G}_{(1)}(\varphi) = 1 - e^{-\gamma}\prod_{n =1}^{\infty} \frac{1 - \alpha_n}{1 + \beta_n} \in [0,1)
\end{equation}
\end{theorem}

\begin{proof}
Suppose $\varphi$ is a $\overline{G}$-positive homomorphism of $\Gamma$ with ${G}_{(1)}(\varphi) \ne 1$. By Lemma~\ref{gine} we get 
\begin{equation}\label{gdec}
1 > {G}_{(1)}(\varphi) \ge {G}_{(2)}(\varphi) \ge \ldots \ge 0
\end{equation}
Recall that we have the following infinite positive expansion of $s_{\lambda}$ in the `basis' $\{{G}_{\mu} \}$ which follows from \eqref{gschur} (see \cite{lenart}):
\begin{equation}\label{sog}
s_{\lambda} = \sum_{\mu \supset \lambda} \mathrm{f}_{\mu/\lambda}\, {G}_{\mu},
\end{equation}
where $\mathrm{f}_{\mu/\lambda}$ is the number of elegant tableaux of shape $\mu/\lambda$ (see Def.~\ref{eleg}). Therefore, if for all $\lambda$
$$
0 \le \sum_{\mu \supset \lambda} \mathrm{f}_{\mu/\lambda}\, {G}_{\mu}(\varphi) < \infty
$$
then we can apply the homomorphism $\varphi$ to define $\rho = \rho(\varphi) : \Lambda \to \mathbb{R}$ by letting
\begin{equation}\label{rofi}
\rho(s_{\lambda}) := \sum_{\mu \supset \lambda} \mathrm{f}_{\mu/\lambda}\, {G}_{\mu}(\varphi)
\end{equation}
which becomes a well-defined Schur-positive specialization. 
Let us consider the identity (follows by Lemma~\ref{lhb})
$$
h_{1} = \frac{{G}_{(1)} - {G}_{(2)}}{1 - {G}_{(1)}} = \sum_{\mu} \mathrm{f}_{\mu/(1)}\, {G}_{\mu}
$$
and the expansion
\begin{equation}\label{ham}
(h_{1})^n = \left(\frac{{G}_{(1)} - {G}_{(2)}}{1 - {G}_{(1)}} \right)^n = \sum_{\mu} a_{\mu}\, {G}_{\mu}, \quad a_{\mu} \ge 0.
\end{equation}
Since the function $(h_{1})^n - s_{\lambda}$ is Schur-positive, we must have $a_{\mu} \ge \mathrm{f}_{\mu/\lambda}$ for all $\mu$. Applying the homomorphism $\varphi$ to \eqref{ham} we obtain that the series
$$
\sum_{\mu} a_{\mu}\, {G}_{\mu}(\varphi) = \left(\frac{{G}_{(1)}(\varphi) - {G}_{(2)}(\varphi)}{1 - {G}_{(1)}(\varphi)} \right)^n
$$
converges since ${G}_{(1)}(\varphi) < 1$. Therefore,
$$
0 \le \sum_{\mu \supset \lambda} \mathrm{f}_{\mu/\lambda}\, {G}_{\mu}(\varphi) \le \sum_{\mu} a_{\mu}\, {G}_{\mu}(\varphi) < \infty
$$
as needed. So we have defined a Schur-positive specialization given by \eqref{rofi}. Let us check its values on the generators $h_n$:
\begin{equation}\label{hnn}
\rho(h_n) = \sum_{\mu \supset (n)} \mathrm{f}_{\mu/(n)}\, {G}_{\mu}(\varphi) = \frac{{G}_{(n)}(\varphi) - {G}_{(n+1)}(\varphi)}{1 - {G}_{1}(\varphi)}
\end{equation}
where the last identity makes sense as it is an identity for $\{G_{\mu}\}$ (see Lemma~\ref{lhb}) and ${G}_{1}(\varphi) < 1$. 

Let us now show that the series 
$$
\rho(H(1)) = 1 + \sum_{n = 1}^{\infty} \rho(h_n) 
$$
converges. By \eqref{gdec} there exists $\lim_{N \to \infty} G_{(N)}(\varphi) = A \in [0,1]$. Using this and \eqref{hnn} we have 
$$
\rho(H(1)) = \lim_{N \to \infty} \left(1 + \sum_{n = 1}^{N} \rho(h_n)\right) = \lim_{N \to \infty} \frac{1 - {G}_{(N+1)}(\varphi)}{1 - {G}_{1}(\varphi)} = \frac{1 - A}{1 - {G}_{1}(\varphi)}
$$
On the other hand, using the identity
$$
H(1) = \frac{1}{1 - {G}_{(1)}} = \sum_{k = 0}^{\infty} ({G}_{(1)})^k
$$
we obtain  
$$
\rho(H(1)) = \frac{1}{1 - {G}_{(1)}(\varphi)}.
$$
Furthermore, since $\rho$ is Schur-positive, it is parametrized by some nonnegative reals $\{\alpha_n\}$, $\{\beta_n \}$, $\gamma$, and we have
$$
1 \le \rho(H(1)) = \frac{1}{1 - {G}_{(1)}(\varphi)} = e^{\gamma} \prod_{n =1}^{\infty} \frac{1 + \beta_n}{1 - \alpha_n}  < \infty
$$ 
or 
$$
{G}_{(1)}(\varphi) = 1 - e^{-\gamma}\prod_{n =1}^{\infty} \frac{1 - \alpha_n}{1 + \beta_n} \in [0,1)
$$
as needed.
\end{proof}


\subsection{Structure of $\overline{G}$-positive specializations} 

\begin{lemma}\label{lll}
Let $\alpha, \beta, \gamma \in \mathbb{R}_{\ge 0}$ and $\alpha \le 1$. Then the Schur-positive generators of $\Lambda$ extend to $\overline{G}$-positive specializations $\hat\phi_{\alpha},$ $\hat\varepsilon_{\beta}$, $\hat\pi_{\gamma}$ of $\Gamma$. 
\end{lemma}
\begin{proof}
First, using the single variable substitution \eqref{single} we have 
$${G}_{\lambda/\!\!/\mu}(\hat\phi_{\alpha}) = \alpha^{|\lambda/\mu|}(1 - \alpha)^{a(\lambda/\!\!/\mu)} \ge 0.$$  

Next, for the specialization $\hat\varepsilon_{\beta}$ we have
$${G}_{\lambda/\!\!/\mu}(\hat\varepsilon_{\beta}) = {G}_{\lambda'/\!\!/\mu'}(\hat\phi_{\beta'}) = \beta^{|\lambda'/\mu'|} / (1 + \beta)^{|\lambda'/\mu'| + a(\lambda'/\!\!/\mu')} \ge 0, \text{ where } \beta' = \beta/(1 + \beta).$$ 
Here we used the formula  (see \cite{dy})
$$\omega({G}_{\lambda/\!\!/\mu}(x_1, x_2, \ldots)) = {G}_{\lambda'/\!\!/\mu'}(x_1/(1 + x_1), x_2/(1 + x_2), \ldots).$$ 

For the Plancherel specialization $\hat\pi_{\gamma},$ let us first show that the power series ${G}_{\lambda/\!\!/\mu}(\hat\pi_{\gamma}) \in \mathbb{R}[[\gamma]]$ converges for all $\gamma \in \mathbb{R}$. From the definition of $G$ functions we have $${G}_{\lambda/\!\!/\mu}(\hat\pi_{\gamma}) = (-1)^{|\lambda/\mu|} \tilde {G}_{\lambda/\!\!/\mu}(\hat\pi_{- \gamma})$$ and since $\tilde {G}_{\lambda/\!\!/\mu}(\hat\pi_{- \gamma}) \in \mathbb{R}[[\gamma]]$ converges for all $\gamma \in \mathbb{R}$ (see proof of Lemma~\ref{sext}), the same holds for ${G}_{\lambda/\!\!/\mu}(\hat\pi_{\gamma})$. To see nonnegativity, notice that $\pi_{\gamma}$ can be realized as follows: $${G}_{\lambda/\!\!/\mu}(\hat\pi_{\gamma}) = \lim_{N \to \infty} {G}_{\lambda/\!\!/\mu}(\underbrace{\gamma/N, \ldots, \gamma/N}_{N \text{ times}}) = \lim_{N \to \infty} {G}_{\lambda/\!\!/\mu}(\underbrace{\hat\phi_{\gamma/N}, \ldots, \hat\phi_{\gamma/N}}_{N \text{ times}})$$
for $\gamma/N \le 1$ we have $\hat\phi_{\gamma/N}$ is $\overline{G}$-positive and so the union $(\hat\phi_{\gamma/N}, \ldots, \hat\phi_{\gamma/N})$ is $\overline{G}$-positive as well. Hence ${G}_{\lambda/\!\!/\mu}(\hat\pi_{\gamma}) \ge 0$ as desired.
\end{proof}

\begin{proposition}
Let $\rho$ be a Schur-positive specialization of $\Lambda$ with nonnegative real parameters 
$\{\alpha_n < 1\}$, $\{\beta_n \}$, $\gamma$ such that
\begin{equation}\label{egamma}
\delta := 1 - e^{-\gamma}\prod_{n =1}^{\infty} \frac{1 - \alpha_n}{1 + \beta_n} \in [0,1)
\end{equation}
Then it extends to a $\overline{G}$-positive homomorphism $\hat\rho$ of $\Gamma$ that is a union of Schur-positive generators $\hat\pi_{\gamma},$ $(\hat\phi_{\alpha_n}),$ $(\hat\varepsilon_{\beta_n})$. 
In particular, ${G}_{(1)}(\hat\rho) = \delta$
and the generating function for the elements $\{ {G}_{(n)}(\hat\rho)\}$ is the following
\begin{equation}\label{eeq3}
1 + (z - 1) \sum_{n = 1}^{\infty} {G}^{}_{(n)}(\hat\rho) z^{n - 1} = e^{\gamma (z - 1)} \prod_{n = 1}^{\infty} \frac{1 - \alpha_n}{1 - \alpha_n z} \prod_{n = 1}^{\infty} \frac{1 + \beta_n z }{1 + \beta_n}.
\end{equation}
\end{proposition}
\begin{proof}
By the factorization form in Theorem~\ref{etfact}, $\rho$ as a union of Schur-positive generators $\pi_{\gamma}$, $(\phi_{\alpha_n})$, $(\varepsilon_{\beta_n})$ (see subsec.~\ref{sgen}). Then these Schur-positive generators extend to $\overline{G}$-positive specializations by Lemma~\ref{lll}. 
Define $\varphi$ as the union of positive generators $\pi_{\gamma}, (\hat\phi_{\alpha_n}), (\varepsilon_{\beta_n})$. Then the formulas \eqref{egamma}, \eqref{eeq3} follow by Schur expansions of $\{G_{{n}}\}$, see Lemma~\ref{lhb}. 
\end{proof}

\subsection{The normalized case ${G}_{(1)}(\varphi) = 1$} Recall that $G_{(1)} = 1 - \prod_{n}(1 - x_n)$ and hence a normalized homomorphism of $\Gamma$ maps $\prod_{n}(1 - x_n)$ to $0$.

\begin{lemma}\label{buchl}
We have ${G}_{\lambda/\!\!/\mu}(1, \mathbf{x}) = {G}_{\widetilde{\lambda}/\!\!/\mu}(\mathbf{x})$, where $\widetilde{\lambda} = (\lambda_2, \lambda_3, \ldots)$, i.e. the first row in $\lambda$ is removed.
\end{lemma}
\begin{proof}
By the branching formula \eqref{br} we have ${G}_{\lambda/\!\!/\mu}(1, \mathbf{x}) = \sum_{\nu} {G}_{\lambda/\!\!/\nu}(1)\, {G}_{\nu/\!\!/\mu}(\mathbf{x})$. From the single variable formula \eqref{single} it is not hard to see that we have for $\nu \ne \varnothing$
$$
{G}_{\lambda/\!\!/\nu}(1) = 
\begin{cases}
	1, & \text{ if } \nu = \widetilde\lambda,\\
	0, & \text{ otherwise}
\end{cases}
$$
which implies the needed.
\end{proof}

\begin{theorem}
Let $\varphi : \Gamma \to \mathbb{R}$ be a homomorphism and $\varphi' = (1, \varphi)$ (where $1 = \hat\phi_{1}$). Then: 
\begin{itemize}
\item[(i)] $\varphi'$ is a normalized homomorphism of $\Gamma$. 

\item[(ii)] $\varphi$ is $\overline{G}$-positive if and only if $\varphi'$ is $\overline{G}$-positive.
\end{itemize}
\end{theorem}
\begin{proof}
(i) $\varphi'$ is normalized since $G_{(1)}(1, \varphi') = 1 - (1 - 1) G_{(1)}(\varphi) = 1.$

(ii) If $\varphi$ is $\overline{G}$-positive, then $\varphi'$ is also $\overline{G}$-positive since it is a union of $\overline{G}$-positive specializations $\mathbf{x} \mapsto (1, 0,0, \ldots)$ and $\varphi$. 

Conversely, suppose $\varphi'$ is $\overline{G}$-positive.  
By Lemma~\ref{buchl} we obtain $0 \le {G}_{\lambda/\!\!/\mu}(\varphi') = {G}_{\lambda/\!\!/\mu}(1, \varphi) = {G}_{\widetilde{\lambda}/\!\!/\mu}(\varphi)$ implying that $\varphi$ is $\overline{G}$-positive.
\end{proof}

\begin{remark}
Lemma~\ref{buchl} is a slight extension of Buch's observation \cite{buch} for $\mu = \varnothing$.
\end{remark}

\begin{remark}
Using these properties we also obtain the following $\overline{G}$-positive specializations:
\begin{itemize}
\item[(i)] $\mathbf{x} \mapsto 1^N$ for which ${G}_{\lambda}(1^N) = 1$
if $\ell(\lambda) \le N$ and $0$ otherwise;
\item[(ii)] by letting $N \to \infty$, we get ${G}_{\lambda} \mapsto 1$; in other words, the linear map $\varphi : \Gamma \to \mathbb{R}$ given by $\varphi : {G}_{\lambda} \mapsto 1$ for all $\lambda$ is a well-defined $\overline{G}$-positive homomorphism of $\Gamma$.
\end{itemize}
This implies certain identities for dual families given in \cite{dy4} from combinatorial perspective.
\end{remark}

\section{Dual Grothendieck-positive specializations}\label{dgs}
\begin{definition}
A {\it reverse plane partition} (RPP) of shape $\lambda/\mu$ is a filling of its boxes with positive integers that weakly increase both in rows from left to right and in columns from top to bottom.
Let $RPP(\lambda/\mu)$ be the set of RPP of shape $\lambda/\mu$. For $T \in RPP(\lambda/\mu)$, define the monomial $x^T = \prod_{i \ge 1} x_i^{c_i},$ where $c_i$ is the number of columns that contain $i$. 
Recall also that $c(\lambda/\mu)$ denotes the number of columns of $\lambda/\mu$.
\end{definition}

\begin{definition}[\cite{lp}]
The {\it dual symmetric Grothendieck polynomials} $g^{}_{\lambda/\mu}$ are  defined as follows:
$$
g^{}_{\lambda/\mu} = g^{}_{\lambda/\mu}(x_1, x_2, \ldots) := \sum_{T \in RPP(\lambda/\mu)} 
x^T  
$$
\end{definition}

Note that $g^{}_{\lambda} = s_{\lambda} + \{\text{lower degree elements}\}$ since the largest weight monomials from $\mathrm{RPP}$ correspond to SSYT of shape $\lambda$, and hence $\{g^{}_{\lambda}\}$ is a basis of $\Lambda$. The basis $\{ g^{}_{\lambda}\}$ is in fact {\it dual} to $\{ {G}^{}_{\lambda}\}$ via the Hall inner product for which Schur functions form an orthonormal basis. 

Similarly as for $G_{\lambda}$, there is an involutive automorphism $\hat\tau : \Lambda \to \Lambda$ given on generators by 
$$
\hat\tau : h_{n} \longmapsto \sum_{i = 1}^{n} 
\binom{n - 1}{i-1} e_i 
$$
and for which \cite{dy2}
$$
\hat\tau(g_{\lambda/\mu}) = g_{\lambda'/\mu'}.
$$

\begin{proposition}[\cite{dy2}]
The following branching formula holds
\begin{equation}\label{brg}
	g^{}_{\lambda/\mu}(\mathbf{x}, \mathbf{y}) = \sum_{\mu \subset \nu \subset \lambda} g^{}_{\lambda/\nu}(\mathbf{x}) g^{}_{\nu/\mu}(\mathbf{y})
\end{equation}
For a single variable $x$ we have
\begin{equation}
	g^{}_{\lambda/\mu}(x) = 
		\begin{cases}
			x^{c(\lambda/\mu)}, & \text{ if } \mu \subset \lambda;\\
			0, & \text{ otherwise.}
		\end{cases}
\end{equation}
\end{proposition}

\begin{definition}[Dual Grothendieck-positive specializations]
A homomorphism $\rho : \Lambda \to \mathbb{R}$ is called {\it $g$-positive} if $\rho(g^{}_{\lambda/\mu}) \ge 0$ for all $\lambda, \mu$.
\end{definition}

We describe a class of $g$-positive specializations as follows:
\begin{proposition}\label{ggg}
Let $\rho : \Lambda \to \mathbb{R}$ be a specialization given by 
\begin{equation}\label{eqq1}
\rho (H(z)) = 1+ \sum_{n = 1}^{\infty} \rho(h_n)\, z^n = e^{\gamma z + \delta z/(1 - z)} \prod_{n = 1}^{\infty} \frac{1}{1 - \alpha_n z} \prod_{n = 1}^{\infty} \left(1 + \frac{\beta_n z}{1 -  z} \right)
\end{equation}
for nonnegative reals  
$\{\alpha_n\}$,  $\{\beta_n\}$,  
$\gamma, \delta$ 
such that $\sum_{n} (\alpha_n + \beta_n) < \infty$. Then $\rho$ is $g$-positive.
\end{proposition}
\begin{proof}
First, note that $g_{(n)} = h_n$. 
By the branching formula \eqref{brg}, we obtain that union of  $g$-positive specializations is also $g$-positive. 
Let $\alpha,\beta,\gamma, \delta \ge 0$. The following specializations are all $g$-positive:
\begin{itemize}
\item[(a)] $\phi_{\alpha} : (x_1, x_2, \ldots) \mapsto (\alpha, 0, 0, \ldots)$. 
Here $\phi_{\alpha}(g_{\lambda/\mu}) = \alpha^{c(\lambda/\mu)} \ge 0$ 
and we have 
$$\phi_\alpha(h_{n}) = \alpha^n, \qquad \phi_{\alpha}(H(z)) = \frac{1}{1 - \alpha z}$$
\item[(b)] $\psi_{\beta} := \phi_{\beta} \circ \hat\tau$. Here 
$\psi_{\beta}(g_{\lambda/\mu}) = \beta^{c(\lambda'/\mu')} \ge 0$ 
and we have 
$$\psi_\beta(h_{n}) = \beta, \qquad \psi_{\beta}(H(z)) = 1 + \frac{\beta z}{1 - z}$$
\item[(c)] The {Plancherel specialization} $\pi_{\gamma}$: $p_1 \mapsto \gamma$ and $p_{k} \mapsto 0$ for $k \ge 2$. Since the polynomials $g_{\lambda/\mu}$ are Schur-positive (see \cite{gal}, also \cite{lp} for straight shapes), we have $\pi_{\gamma}(g_{\lambda/\mu}) \ge 0$ and 
here  $$\pi_{\gamma} (H(z)) = e^{\gamma z}.$$
\item[(d)] The {\it dual} Plancherel specialization $\tilde \pi_{\delta} := \pi_{\delta} \circ \hat\tau$. Here $\tilde \pi_{\delta}(g_{\lambda/\mu}) = \pi_{\delta}(g_{\lambda'/\mu'}) \ge 0$ and we have $$\tilde \pi_{\delta}(H(z)) = e^{\delta z/(1 - z)}.$$
\end{itemize}
Now we can define the specialization $\rho$ as the union of specializations $\pi_{\gamma}$, $\tilde\pi_{\delta}$, $(\phi_{\alpha_n})$, $(\psi_{\beta_n})$ which gives $g$-positivity. 
\end{proof}

We conjecture that the converse is also true.
\begin{conjecture}
Every $g$-positive specialization is characterized by \eqref{eqq1}.
\end{conjecture}

\begin{remark}
Equivalently, $g$-positive specializations $\rho$ in Prop.~\ref{ggg} are given by 
\begin{align*}
&\rho : p_1 \longmapsto \gamma + \delta + \sum_{n} (\alpha_n + \beta_n)\\
&\rho : p_k \longmapsto  \delta k + \sum_{n} \left(\alpha_n^k + \sum_{\ell = 1}^{k} (-1)^{\ell} \binom{k - 1}{\ell - 1} \beta_n^{\ell}\right) \qquad k \ge 2.
\end{align*}
\end{remark}
\begin{remark}
The set of $g$-positive specializations is {\it larger} than the set of Schur-positive specializations. Since the functions $g_{\lambda/\mu}$ are Schur-positive, every Schur-positive specialization is $g$-positive as well. However the converse is not true. For example, take the $g$-positive specialization $\rho = \psi_{\beta} : h_n \mapsto \beta \in  (0,1)$ for all $n \ge 1$. Then $\rho(s_{(1^2)}) = \rho(e_2) = \rho(h_{1}^{2} - h_2) = \beta^2 - \beta < 0$.
\end{remark}

\section{Two analogues of the Plancherel measure on partitions}
\subsection{Corner growth model}
Define a sequence of random partitions (or more concretely, their corresponding Young diagrams) $(\lambda^{(1)}, \lambda^{(2)}, \ldots)$ such that $\lambda^{(1)} = (1)$ and $\lambda^{(n+1)}$ is obtained from $\lambda^{(n)}$ by adding a box in one of its {\it outer corners} with equal probability, i.e. with probability $1/\#\{\text{outer corners}\}$. This Markov process is known as the {\it corner growth model} which can be viewed as a TASEP (see e.g. \cite{romik}). 
Let 
$$
p_n(\lambda) := \mathrm{Pr}(\lambda^{(n)} = \lambda). 
$$
As $n \to \infty$, the sequence $\{ \lambda^{(n)}\}$ rescaled by $1/\sqrt{n}$ has the following parabolic limit shape
\footnote{ Formally, for any $\epsilon \in (0,1)$
$$
\mathrm{Pr}\left((1 - \epsilon)\omega \subset [\lambda^{(n)}] \subset (1 + \epsilon)\omega \right) \to 1 \text{ as } n \to \infty,
$$
where $\omega = \{(x,y) : x,y \ge 0, \sqrt{x} + \sqrt{y} \le 6^{1/4} \}$ and $[\lambda^{(n)}]$ is the region under graph of partitions $\lambda^{(n)}$ rescaled by $1/\sqrt{n}$ in both directions and drawn in French notation (upside down English).
For more background on the corner growth model, its analysis, and connections with RSK, see \cite{romik}.}
$$
\sqrt{x} + \sqrt{y} = 6^{1/4}.
$$
This result was first proved by Rost \cite{rost}. 
The model was further analyzed by Johansson \cite{joh1} (cf. \cite{joh2}) making connections with RSK and {\it longest increasing subsequences} in generalized permutations.

Interestingly, this process is related to dual Grothendieck polynomials as follows. 
Let $\rho : \Lambda \to \mathbb{R}$ be a normalized $g$-positive specialization, i.e. $\rho(g_{(1)}) = 1$. 
Then it is not difficult to show that $$\mu_{\rho,n}(\lambda) := p_n(\lambda)\, \rho(g_{\lambda})$$ is a probability measure on the set of partitions $\lambda \vdash n$. 

In particular, take the specialization $\rho(g_{\lambda}) = g_{\lambda}(1) = 1$, then as $n \to \infty$, partitions with respect to the measure $\mu_{\rho,n}(\lambda) = p_n(\lambda)$ have the above parabolic limit shape. 
Now we pose the same question for the Plancherel specialization $\rho = \pi$: $p_1 \mapsto 1$ and $p_k \mapsto 0$ for $k \ge 2$. What is limit shape of partitions with respect to the measure
$\mu_{\rho,n}(\lambda) = p_n(\lambda)\, \pi(g_{\lambda})$?
\begin{remark}
The probabilities $p_n(\lambda)$ serve here as analogues of dimensions $f^{\lambda}$ for the classical Schur case with the Plancherel measure $f^{\lambda}\, \pi(s_{\lambda}) = (f^{\lambda})^2 /{n!}$.
\end{remark}
\begin{remark}
We address some further connections between the corner growth model and dual Grothendieck polynomials in \cite{dy3}.
\end{remark}
\subsection{Plancherel-Hecke measure} 
Let us introduce two more types of tableaux: 

An {\it increasing tableau} is a filling of a Young diagram with positive integers strictly increasing both in rows and columns. Denote by $d^{\lambda}(n)$ the number of increasing tableaux of shape $\lambda$ filled with numbers from the set $\{1, \ldots, n \}$.  

A {\it standard set-valued tableau} (SSVT) is a filling of a Young diagram with sets of positive integers so that entries strictly increase in both  rows and columns. Denote by $e^{\lambda}(m)$ the number of SSVT filled with numbers $\{1, \ldots, m \}$.

Let $\varphi : \Gamma \to \mathbb{R}$ be a $G$-positive specialization of $\Gamma$. Then it is not difficult to show that 
$$
M_{\varphi, n}(\lambda) := \frac{d^{\lambda}(n)\, \tilde G_{\lambda}(\varphi)}{\Delta^n}, \qquad \Delta := 1 + \tilde G_{(1)}(\varphi)
$$
is a probability measure on the set of partitions $\lambda  \subset \delta_n := (n, n-1, \ldots, 1)$.

In particular, take the 
Plancherel specialization $\varphi = \hat\pi_\gamma$. Then 
we can derive that 
$$
\tilde G_{\lambda}(\hat\pi_\gamma) = \sum_{m} \frac{\gamma^m}{m!} \, e^{\lambda}(m)
$$
and the measure $M_{\varphi, n}(\lambda)$ specializes to the {\it Plancherel-Hecke measure} $\mu_{m,n}(\lambda)$ studied by Thomas and Yong in \cite{thomasyong}. 
Namely, we have 
$$\mu_{m,n}(\lambda) = \frac{d^{\lambda}(n)\, e^{\lambda}(m)}{n^m}$$ is a probability measure on the set of partitions $\lambda \subset \delta_n$ and $|\lambda| \le m$. This result was obtained from a $K$-theoretic extension of RSK, the {\it Hecke insertion algorithm} \cite{bksty, thomasyong2}. 
The measure is also naturally related to {longest increasing subsequences} of words of length $m$ in the alphabet $\{1, \ldots, n \}$. A special case of the conjecture made in \cite{thomasyong} is that when $n = \Theta(m^{\alpha})$ for $\alpha > 1/2$, the partitions with respect to this measure have the same limit shape as for the Plancherel measure 
obtained by Vershik and Kerov \cite{vkp}, and Logan and Shepp \cite{ls}.  

\section*{Acknowledgements}
I am grateful to Askar Dzhumadil'daev, Igor Pak, Leonid Petrov, and Pavlo Pylyavskyy for many helpful conversations. I am also grateful to the referee for helpful remarks.


\end{document}